\documentclass[12pt]{amsart}
\usepackage{amssymb,array,hyperref}

\newtheorem{theorem}{Theorem}[section]
\newtheorem{definition}[theorem]{Definition}
\newtheorem{conjecture}[theorem]{Conjecture}
\newtheorem{proposition}[theorem]{Proposition}

\newtheorem{problem}[theorem]{Problem}
\newtheorem{lemma}[theorem]{Lemma}

\begin{document}

\title{Representations of quantum permutation algebras}

\author{Teodor Banica}
\address{T.B.: Department of Mathematics, Toulouse 3 University, 118 route de Narbonne, 31062 Toulouse, France. {\tt banica@math.ups-tlse.fr}}

\author{Julien Bichon}
\address{J.B.: Department of Mathematics, Clermont-Ferrand 2 University, Campus des Cezeaux, 63177 Aubiere Cedex, France. {\tt bichon@math.univ-bpclermont.fr}}

\author{Jean-Marc Schlenker}
\address{J.-M.S.: Department of Mathematics, Toulouse 3 University, 118 route de Narbonne, 31062 Toulouse, France. {\tt schlenker@math.ups-tlse.fr}}

\thanks{The work of T.B. and J.B. was supported by the ANR grant ``Galoisint''}

\begin{abstract}
We develop a combinatorial approach to the quantum permutation algebras, as Hopf images of representations of type $\pi:A_s(n)\to B(H)$. We discuss several general problems, including the commutativity and cocommutativity ones, the existence of tensor product or free wreath product decompositions, and the Tannakian aspects of the construction. The main motivation comes from the quantum invariants of the complex Hadamard matrices: we show here that, under suitable regularity assumptions, the computations can be performed up to $n=6$.
\end{abstract}

\maketitle

\section*{Introduction}

The free analogue of the symmetric group $S_n$ was constructed by Wang in \cite{wan}. The idea is that when regarding $S_n$ as a complex algebraic group, the $n\times n$ matrix formed by the standard coordinates $u_{ij}:S_n\to\mathbb C$ is magic, in the sense that all its entries are projections, which sum up to 1 on each row and each column. So, Wang considers then the universal algebra $A_s(n)$ generated by the entries of an abstract $n\times n$ magic matrix. This is a Hopf algebra in the sense of Woronowicz \cite{wo1}, so its spectrum $S_n^+$ is a compact quantum group, called quantum permutation group.

The very first question is whether the ``quantum permutations'' do exist or not. That is, we would like to know whether $S_n^+$ is indeed bigger that $S_n$, and if so, how big is it. Or, in other words, if $A_s(n)$ is bigger than $C(S_n)$, and if so, how big is it.

The answer to these basic questions is as follows:
\begin{enumerate}
\item At $n\leq 3$ we have $S_n^+=S_n$. This is because the entries of such a $n\times n$ magic matrix can be shown to pairwise commute, so we have $A_s(n)=C(S_n)$.

\item At $n=4$ we have $S_4^+=SO_3^{-1}$. This is a quite subtle result, the quantum group $S_4^+$ being in fact the central object of the whole theory. See \cite{ba2}, \cite{bb2}, \cite{bc2}.

\item At $n\geq 5$ the situation is even worse: the dual of $S_n^+$ is not amenable, and there is indication from \cite{vve} that its reduced group algebra should be simple.
\end{enumerate}

The world of quantum permutation groups, i.e. quantum subgroups of $S_n^+$, turns to be extremely rich. For instance it was shown in \cite{ba2} that these quantum groups are in correspondence with the subalgebras of Jones' spin planar algebra \cite{jo2}. Another key result in this sense is the one in \cite{bb2}, where a complete classification is obtained at $n=4$. The computation of integrals over the quantum permutation groups gives rise to a subtle problematics, of theoretical physics flavor \cite{bc1}, \cite{bc2}. Some new connections with noncommutative geometry and with free probability were found in \cite{bgs}, \cite{ksp}.

An important class of examples, which actually motivated the whole theory, comes from the complex Hadamard matrices. These are the $n\times n$ matrices formed by complex numbers of modulus 1, whose rows are pairwise orthogonal.

The point is that each Hadamard matrix $h\in M_n(\mathbb C)$ produces a quantum permutation algebra, i.e. a quotient $A_s(n)\to A$, according to the following algorithm:
\begin{enumerate}
\item We know that the rows $h_i\in\mathbb C^n$ are pairwise orthogonal.

\item Thus the vectors $\xi_{ij}=h_i/h_j$ form a magic basis of $\mathbb C^n$.

\item This gives a representation $\pi:A_s(n)\to M_n(\mathbb C)$.

\item We call $A$ the Hopf image of this representation.
\end{enumerate}

The basic example comes from the Fourier matrix, $F_{ij}=w^{(i-1)(j-1)}$ with $w=e^{2\pi i/n}$. All the above objects are ``circulant'', and we end up with the algebra $A=C(\mathbb Z_n)$.

The above construction has been known for about 10 years, since \cite{ba1}. Its basic properties were worked out in the recent paper \cite{bni}. The notion of Hopf image was systematically investigated in the preprint \cite{bb3}. The reasons for this delayed development is the difficulty in producing non-trivial statements on the subject.

In fact, the various problems regarding the complex Hadamard matrices (classification, computation of invariants) are all reputed to be quite difficult, with the tools basically lacking. The philosophy is somehow that ``the Fourier matrix corresponds to the known mathematics, and the other matrices correspond to unknown mathematics''. Illustrating here is the classification work of Haagerup \cite{ha1}, the work on invariants by Jones \cite{jo2}, as well as a counterexample constructed by Tao in \cite{tao}.

Let us mention for instance that one particularly difficult problem, well-known to specialists, is the computation of the quantum invariants of the following $7\times 7$ matrix based on the root of unity $w=e^{2\pi i/6}$, discovered by Petrescu in \cite{pet}:
$$P^q
=\begin{pmatrix}
1&1&1&1&1&1&1\\
1&qw&qw^4&w^5&w^3&w^3&w\\
1&qw^4&qw&w^3&w^5&w^3&w\\
1&w^5&w^3&\bar{q}w&\bar{q}w^4&w&w^3\\
1&w^3&w^5&\bar{q}w^4&\bar{q}w&w&w^3\\
1&w^3&w^3&w&w&w^4&w^5\\
1&w&w&w^3&w^3&w^5&w^4
\end{pmatrix}$$

The purpose of the present paper is to develop a systematic study of the representations of type $\pi:A_s(n)\to B(H)$, where $H$ is a Hilbert space. Besides the above-mentioned Hadamard matrix motivation, we have as well an abstract motivation: any quantum permutation algebra appears as Hopf image of such a representation.

So, let us consider a representation of type $\pi:A_s(n)\to B(H)$, and let $A$ be its Hopf image. We have the following list of basic questions:
\begin{enumerate}
\item When is $A$ commutative?

\item When is $A$ cocommutative?

\item Do we have $A=A'\otimes A''$?

\item Do we have $A=A'*_wA''$?
\end{enumerate}

We will discuss all these questions, with a particular attention to the case $H=\mathbb C^n$, which includes the Hadamard matrix situation. We will discuss as well the classification problem for $\pi$ and the explicit computation of $A$, for small values of $n$.

Our study will lead naturally to a certain hierarchy for the related combinatorial objects associated to Hilbert spaces. In decreasing order of generality, these are:
\begin{center}
\begin{tabular}[t]{|l|l|l|l|l|}
\hline Object&Classification&Hopf algebra computation\\
\hline
\hline Magic decompositions&$n\leq 3$ done, $n=4$ difficult&$n\leq 3$ done, $n=4$ difficult\\
\hline Magic bases&$n\leq 3$ done, $n=4$ possible&$n\leq 3$ done, $n=4$ possible\\
\hline Hadamard matrices&$n\leq 5$ done, $n=6$ difficult&$n\leq 5$ done, $n=6$ difficult\\
\hline Regular Hadamard&$n\leq 6$ done, $n=7$ possible&$n\leq 5$ done, $n=6$ possible\\
\hline
\end{tabular}
\end{center}
\medskip

The precise content of this table will be explained in the body of the paper.
 
The above hierarchy is quite natural, with the study of the regular matrices being related to some key problems. In fact, our main results concern precisely the regular matrices: at $n=6$ we already have a quite satisfactory picture, and at $n=7$, which would be the next step, we have the above-mentioned Petrescu matrix.  

Let us also mention that another motivation for the study of the regular matrices, and of their one-parameter deformations over the unit circle, would be the development of an abstract theory of ``quantum permutation groups at roots of unity''. Observe that, unlike for the quantized enveloping algebras of Drinfeld \cite{dri} and Jimbo \cite{jim}, in our case the square of the antipode is always the identity: $S^2=id$.

The paper is organized as follows. In 1 we recall the construction of the Wang algebra, in 2-4 we discuss the general properties of its Hilbert space representations, and in 5-6 we focus on the representations coming from complex Hadamard matrices. In 7-10 we present a number of technical results regarding the Hadamard matrices of small order, and in 11 we derive from this study several classification results.

The final section, 12, contains a few concluding remarks.

\subsection*{Acknowledgements}

This paper was written over a long period of time, and has benefited from several discussions with our colleagues, which succesively reshaped the organization and goals of the manuscript. We would like in particular to thank R. Burstein, B. Collins, V. Guedj, V. Jones, R. Nicoara, F. Sz\"oll\H{o}si and S. Vaes.

\section{Quantum permutations}

Let $A$ be a $C^*$-algebra. That is, we have a complex algebra with a norm and an involution, such that the Cauchy sequences converge, and $||aa^*||=||a||^2$.

The basic example is $B(H)$, the algebra of bounded operators on a Hilbert space $H$. In fact, any $C^*$-algebra appears as closed subalgebra of some $B(H)$.

The key example is $C(X)$, the algebra of continuous functions on a compact space $X$. By a theorem of Gelfand, any commutative $C^*$-algebra is of the form $C(X)$.

There are several ways of passing from commutative $C^*$-algebras to noncommutative ones. In this paper we use an approach based on the notion of projection.

\begin{definition}
Let $A$ be a $C^*$-algebra.
\begin{enumerate}
\item A projection is an element $p\in A$ satisfying $p^2=p=p^*$.

\item Two projections $p,q\in A$ are called orthogonal when $pq=0$. 

\item A partition of unity is a set of orthogonal projections, which sum up to $1$.
\end{enumerate}
\end{definition}

In the case of the above two basic examples, these notions are as follows.

A projection in $B(H)$ is an orthogonal projection $P_K$, where $K\subset H$ is a closed subspace. The orthogonality of projections corresponds to the orthogonality of subspaces, and the partitions of unity correspond to the orthogonal decompositions of $H$.

A projection in $C(X)$ is a characteristic function $\chi_Y$, where $Y\subset X$ is an open and closed subset. The orthogonality of projections corresponds to the disjointness of subsets, and the partitions of unity correspond to the partitions of $X$.

The following key definition is due to Wang \cite{wan}.

\begin{definition}
A magic unitary over a $C^*$-algebra $A$ is a square matrix of projections $u\in M_n(A)$, all whose rows and columns are partitions of the unity.
\end{definition}

In the case of the above two basic examples, the situation is as follows.

A magic unitary over $B(H)$ is of the form $P_{K_{ij}}$, with $K$ magic decomposition of $H$, in the sense that all rows and columns of $K$ are orthogonal decompositions of $H$.

A magic unitary over $C(X)$ is of the form $\chi_{Y_{ij}}$, with $Y$ magic partition of $X$, in the sense that all rows and columns of $Y$ are partitions of $X$.

Consider now the situation $G\curvearrowright X$ where a finite group acts on a finite set. The sets $G_{ij}=\left\{\sigma\in G\mid \sigma(j)=i\right\}$ form a magic partition of $G$, so the corresponding characteristic functions form a magic unitary over the algebra $A=C(G)$.

\begin{definition}
The matrix of characteristic functions
$$\chi_{ij}=\chi\left\{\sigma\in G\mid \sigma(j)=i\right\}$$
is called magic unitary associated to $G\curvearrowright X$.
\end{definition}

The interest in $\chi$ is that it encodes the dual structural maps of $G\curvearrowright X$. Consider indeed the multiplication, unit, inverse and action map:
\begin{eqnarray*}
m(\sigma,\tau)&=&\sigma\tau\\
u(\cdot)&=&1\\
i(\sigma)&=&\sigma^{-1}\\
a(i,\sigma)&=&\sigma(i)
\end{eqnarray*}

The duals of these maps are called comultiplication, counit, antipode and coaction. They are given by the following well-known formulae, see \cite{abe}:
\begin{eqnarray*}
\Delta(f)&=&(\sigma,\tau)\to f(\sigma\tau)\\
\varepsilon(f)&=&f(1)\\
S(f)&=&\sigma\to f(\sigma^{-1})\\
\alpha(f)&=&(i,\sigma)\to f(\sigma(i))
\end{eqnarray*}

These latter maps can all be expressed in terms of $\chi$, and in the particular case of $G=S_n$ acting on $X_n=\{1,\ldots ,n\}$, we have the following presentation result.

\begin{theorem}
$C(S_n)$ is the universal commutative $C^*$-algebra generated by $n^2$ elements $\chi_{ij}$, with relations making $(\chi_{ij})$ a magic unitary matrix. The maps
\begin{eqnarray*}
\Delta(\chi_{ij})&=&\sum\chi_{ik}\otimes \chi_{kj}\\
\varepsilon(\chi_{ij})&=&\delta_{ij}\\
S(\chi_{ij})&=&\chi_{ji}\\
\alpha(\delta_i)&=&\sum\delta_j\otimes \chi_{ji}
\end{eqnarray*}
are the comultiplication, counit, antipode and coaction of $C(S_n)\curvearrowright C(X_n)$.
\end{theorem}

\begin{proof}
Let $A$ be the universal algebra in the statement. The Stone-Weierstrass theorem shows that the entries of the magic unitary associated to $S_n\curvearrowright X_n$ generate the algebra $C(S_n)$, so we have a surjective morphism of algebras $A\to C(S_n)$.

It follows from the universal property of $A$ that the maps $\Delta,\varepsilon,S,\alpha$ as in the statement exist. Thus $A$ is a Hopf $C^*$-algebra coacting faithfully on $X_n$, so its spectrum is a subgroup of $S_n$, and by dualizing we obtain the missing arrow $C(S_n)\to A$.
\end{proof}

We can proceed now with liberation. The idea is to remove commutativity from the above considerations. The following key definition is due to Wang \cite{wan}.

\begin{definition}
$A_s(n)$ is the universal $C^*$-algebra generated by $n^2$ elements $u_{ij}$, with relations making $(u_{ij})$ a magic unitary matrix. The maps
\begin{eqnarray*}
\Delta(u_{ij})&=&\sum u_{ik}\otimes u_{kj}\\
\varepsilon(u_{ij})&=&\delta_{ij}\\
S(u_{ij})&=&u_{ji}\\
\alpha(\delta_i)&=&\sum \delta_j\otimes u_{ji}
\end{eqnarray*}
are the comultiplication, counit, antipode and coaction of $A_s(n)\curvearrowright C(X_n)$.
\end{definition}

The algebra $A_s(n)$ is a Hopf $C^*$-algebra in the sense of Woronowicz \cite{wo1}. Its spectrum $S_n^+$ is a compact quantum group, called quantum permutation group on $n$ points.

\begin{theorem}
The algebras $A_s(n)$ are as follows:
\begin{enumerate}
\item For $n\leq 3$, the canonical map $A_s(n)\to C(S_n)$ is an isomorphism.

\item For $n\geq 4$, $A_s(n)$ is not commutative, and infinite dimensional.
\end{enumerate}
\end{theorem}

\begin{proof}
This follows from the fact that the entries of a $n\times n$ magic unitary with $n\leq 3$ have to commute with each other, while at $n\geq 4$ these don't necessarily commute with each other, and can generate an infinite dimensional algebra. See Wang \cite{wan}.
\end{proof}

In terms of quantum groups, for $n\leq 3$ the canonical inclusion $S_n\subset S_n^+$ is an isomorphism, while for $n\geq 4$ the quantum group $S_n^+$ is not classical, nor finite. 

We are now in position of introducing the arbitrary quantum permutation algebras. These are by definition the Hopf algebra quotients of $A_s(n)$.

\begin{definition}
A quantum permutation algebra is a $C^*$-algebra $A$, given with a magic unitary matrix $u\in M_n(A)$, subject to the following conditions:
\begin{enumerate}
\item The elements $u_{ij}$ generate $A$.

\item $\Delta(u_{ij})=\sum u_{ik}\otimes u_{kj}$ defines a morphism $\Delta:A\to A\otimes A$. 

\item $\varepsilon(u_{ij})=\delta_{ij}$ defines a morphism $\varepsilon:A\to \mathbb C$.

\item $S(u_{ij})=u_{ji}$ defines a morphism $S:A\to A^{op}$.
\end{enumerate}
\end{definition}

In what follows, all the quantum permutation algebras will be supposed to be full. This is a technical assumption, not changing the level of generality, stating that $A$ must be the enveloping algebra of the $*$-algebra generated by the elements $u_{ij}$.

If $(A,u)$ and $(B,v)$ are quantum permutation algebras, so are $A\otimes B$ and $A*B$, both taken with the magic unitary $w=diag(u,v)$. See Wang \cite{wan}.

The free wreath product of $(A,u)$ and $(B,v)$ is given by:
$$A*_wB=(A^{*\dim (v)}*B)/<[u_{ij}^{(a)},v_{ab}]=0>$$

Here the exponents on the right refer to the various copies of $A$. We get in this way a quantum permutation algebra, with magic unitary $w_{ia,jb}=u_{ij}^{(a)}v_{ab}$. See \cite{bi1}.

\begin{theorem}
The commutative and cocommutative cases are as follows:
\begin{enumerate}
\item If $G\subset S_n$ is a subgroup then $C(G)$ is a quantum permutation algebra. Any commutative quantum permutation algebra is of this form.

\item If $\mathbb Z_{i_1}*\ldots*\mathbb Z_{i_k}\to\Gamma$ is a quotient group then $C^*(\Gamma)$ is a quantum permutation algebra. Any cocommutative quantum permutation algebra is of this form.
\end{enumerate}
\end{theorem}

\begin{proof}
(1) The first assertion follows from the general considerations in the beginning of this section. The second assertion follows from the Gelfand theorem.

(2) The first assertion follows from the above considerations. Indeed, we have:
$$C^*(\mathbb Z_{i_1}*\ldots*\mathbb Z_{i_k})
\simeq C(\mathbb Z_{i_1})*\ldots*C(\mathbb Z_{i_k})$$

This shows that the algebra on the left is a quantum permutation one, and the same must hold for its quotient $C^*(\Gamma)$. For the second assertion, see \cite{bi2}.
\end{proof}

\section{Hopf images}

In this section we present a purely combinatorial approach to the quantum permutation algebras, in terms of the geometry of subspaces of a given Hilbert space. 

The starting point is the following fundamental result of Gelfand, which was actually at the origins of the whole $C^*$-algebra theory.

\begin{theorem}
Let $\Gamma$ be a discrete group, and $H$ be a Hilbert space. We have a one-to-one correspondence between:
\begin{enumerate}
\item Unitary representations $u:\Gamma\to U(H)$.

\item Representations $\pi:C^*(\Gamma)\to B(H)$.
\end{enumerate}
\end{theorem}

\begin{proof}
Any unitary representation of $\Gamma$ can be extended by linearity to the group algebra $\mathbb C[\Gamma]$, then by continuity to the whole algebra $C^*(\Gamma)$.

Conversely, consider a $C^*$-algebra representation $\pi:C^*(\Gamma)\to B(H)$. The group elements $g\in C^*(\Gamma)$ being unitaries in the abstract sense, their images by $\pi$ must be certain unitaries $u_g\in B(H)$, and this gives the result.
\end{proof}

The above considerations suggest the following definition. 

\begin{definition}
Let $\pi:C^*(\Gamma)\to B(H)$ be a representation.
\begin{enumerate}
\item $\pi$ is called inner faithful if $g\neq h$ implies $\pi(g)\neq\pi(h)$.

\item The Hopf image of $\pi$ is $A_\pi=C^*(\Gamma')$, where $\Gamma'=\pi(\Gamma)$.
\end{enumerate}
\end{definition}

Observe that any faithful representation is inner faithful. The converse is far from being true. For instance in the case $H=\mathbb C^n$, the finite dimensional algebra $M_n(\mathbb C)$ is the target of many inner faithful representations coming from infinite dimensional algebras of type $C^*(\Gamma)$, one for each discrete subgroup $\Gamma\subset U_n$.

We have the following key statement, which provides an abstract characterization for both notions of Hopf image, and inner faithful representation.

\begin{proposition}
Let $\pi:C^*(\Gamma)\to B(H)$ be a representation.
\begin{enumerate}
\item $A_\pi$ is the smallest group algebra realizing a factorization of $\pi$.

\item $\pi$ is inner faithful iff $A=A_\pi$.
\end{enumerate}
\end{proposition}

\begin{proof}
This follows from Theorem 2.1, and from the basic functorial properties of the group algebra construction $\Gamma\to C^*(\Gamma)$.
\end{proof}

We present now an extension of these fundamental notions and results to the case of quantum permutation algebras. Let us first recall that each such algebra satisfies Woronowicz's axioms in \cite{wo1}, so we have the heuristic formula $A=C^*(\Gamma)$, where $\Gamma$ is a discrete quantum group. Thus the above notions and results can be extended, provided that we use the algebra formalism, and make no reference to the underlying discrete quantum groups, which don't exist as concrete objects.

The best is to proceed by converting Proposition 2.3 into a definition.

\begin{definition}
Let $\pi:A\to B(H)$ be a representation.

\begin{enumerate}
\item $A_\pi$ is the smallest quantum permutation algebra realizing a factorization of $\pi$.

\item $\pi$ is called inner faithful if $A=A_\pi$.
\end{enumerate}
\end{definition}

In other words, the Hopf image is the final object in the category of factorizations of $\pi$ through quantum permutation algebras. Both its existence and uniqueness follow from abstract algebra considerations. The idea is that $A_\pi$ can be constructed as being the quotient of $A$ by a suitable ideal, namely the largest Hopf ideal contained in ${\rm Ker}(\pi)$. We refer to \cite{bb3} for full details regarding this construction. 

A first point of interest in the above notions comes from the following result.

\begin{theorem}
Any quantum permutation algebra appears as Hopf image of a representation $\pi:A_s(n)\to B(H)$. Moreover, we can take $H=l^2(\mathbb N)$.
\end{theorem}

\begin{proof}
This follows from the Gelfand-Naimark-Segal theorem, stating that any $C^*$-algebra has a faithful representation on a Hilbert space. Indeed, given an arbitrary quantum permutation algebra $A$, this theorem gives an embedding $j:A\subset B(H)$. 

By composing this embedding with the canonical map $p:A_s(n)\to A$, we get a representation $jp:A_s(n)\to B(H)$. Now since $A$ provides a factorization of $jp$, and is minimal with this property, we conclude that $A$ is the Hopf image of $jp$.

Finally, $A$ being separable, we can take $H$ to be separable, $H=l^2(\mathbb N)$. 
\end{proof}

The above statement reduces in principle the study of the quantum permutation algebras to that of the magic decompositions of Hilbert spaces. Indeed, the representations $\pi:A_s(n)\to B(H)$ are in one-to-one correspondence with the magic unitaries over the algebra $B(H)$, hence with the magic decompositions of $H$.

So, our starting point will be the following definition.

\begin{definition}
A magic decomposition of $H$ is a square matrix of subspaces $X$, all whose rows and columns are orthogonal decompositions of $H$. Associated to $X$ are:
\begin{enumerate}
\item The magic unitary matrix given by $P_{ij}$ =  projection on $X_{ij}$.

\item The representation $\pi:A_s(n)\to B(H)$ given by $\pi(u_{ij})=P_{ij}$. 

\item The quantum permutation algebra $A=A_\pi$ associated to $\pi$.
\end{enumerate}
\end{definition}

We begin our study with the construction of a basic example. Let $H$ be a Hilbert space, given with a decomposition into orthogonal subspaces:
$$H=\bigoplus_{k=1}^NX_k$$

Let also $(E_{ij})$ be a magic partition of the set $I=\{1,\ldots,N\}$, in the sense that all the rows and columns of $E$ are partitions of $I$. We let:
$$X^E_{ij}=\bigoplus_{k\in E_{ij}}X_k$$

It follows from definitions that $X^E$ is a magic decomposition of $H$.

For $k\in\{1,\ldots,N\}$ we denote by $\sigma_k\in S_n$ the permutation given by $\sigma_k(j)=i$ when $k\in E_{ij}$. These permutations $\sigma_1,\ldots,\sigma_N$ uniquely determine $E$. They generate a certain subgroup $G\subset S_n$, than we call group associated to $E$.

\begin{theorem}
For a magic partition decomposition $X^E$ we have $A=C(G)$, where $G\subset S_n$ is the group associated to $E$.
\end{theorem}

\begin{proof}
We will use the basic properties of the Hopf image, for which we refer to \cite{bb3}.

We first review the definition of $G$. We know from Theorem 1.4 that associated to $E$ is a certain representation $\rho:C(S_n)\to C(I)$. This representation is given by $\rho(\chi_{ij})=\chi_{E_{ij}}$, so the corresponding transpose map $r:I\to S_n$ satisfies:
\begin{eqnarray*}
\chi_{ij}(r(k))
&=&\chi_{E_{ij}}(k)\\
&=&\delta_{\sigma_k(j),i}\\
&=&\chi_{ij}(\sigma_k)
\end{eqnarray*}

This gives $r(k)=\sigma_k$ for any $k$, so we can conclude that $G$ is the group generated by the image of $r$. Or, equivalently, that $C(G)$ is the Hopf image of $\rho$. 

We denote by $P_k$ the orthogonal projection onto $X_k$, and by $P_{ij}$ the orthogonal projection onto $X^E_{ij}$. We have:
$$P_{ij}=\sum_{k\in E_{ij}}P_k$$

We claim that the representation of $A_s(n)$ associated to the magic decomposition $X^E$ has a factorization of the following type:
$$\begin{matrix}
A_s(n)&&\to&&M_n(\mathbb C)\\
\ \\
\downarrow&&&&\uparrow\\
\ \\
C(S_n)&\to&C(G)&\to&C(I)
\end{matrix}$$

Indeed, we can define the arrow on the right to be the one given by $\delta_k\to P_k$, and the other 4 arrows, to be the canonical ones. At the level of generators, we have:
$$\begin{matrix}
u_{ij}&&\to&&P_{ij}\\
\ \\
\downarrow&&&&\uparrow\\
\ \\ 
\chi_{ij}&\to&\chi_{ij_{|G}}&\to&\chi_{E_{ij}}
\end{matrix}$$

Thus the above diagram of algebras commutes, as claimed. Now since $C(G)$ is a Hopf algebra, the Hopf algebra $A_\pi$ we are looking for must be a quotient of it. 

On the other hand, $A_\pi$ must be the minimal algebra containing the image of $C(S_n)$ by the bottom map, so we get $A_\pi=C(G)$ as claimed.
\end{proof}

\begin{theorem}
For a magic decomposition $X_{ij}$, the following are equivalent:
\begin{enumerate}
\item $A$ is commutative.

\item $X=X^E$ for a certain magic partition $E$.
\end{enumerate}
\end{theorem}

\begin{proof}
Indeed, if $A$ is commutative, its quotient algebra $B=C^*(P_{ij})$ must be commutative as well. By applying the Gelfand theorem we get an isomorphism $B\simeq C(I)$, where $I$ is a certain finite set. The magic unitary $(P_{ij})$ must correspond in this way to a magic matrix of characteristic functions $(\chi_{ij})$, which should come in turn from a magic partition $(E_{ij})$ of the set $I$. This gives the result.
\end{proof}

We discuss now the classification of small order magic decompositions, and the computation of the associated Hopf algebras. We fix a Hilbert space $H$.

\begin{theorem}
The $2\times 2$ magic decompositions of $H$ are of the form
$$X=\begin{pmatrix}
A&B\\
B&A
\end{pmatrix}$$
with $H=A\oplus B$. The associated Hopf algebra is $C(G)$, with $G\in\{1,\mathbb Z_2\}$.
\end{theorem}

\begin{proof}
First, it follows from definitions that $X$ must be of the above form. Since the algebra generated by the projections onto $A,B$ is of dimension 1 or 2, depending on whether one of $A,B$ is 0 or not, this gives the second assertion.
\end{proof}

\begin{theorem}
The $3\times 3$ magic decompositions of $H$ are of the form
$$X=\begin{pmatrix}
A\oplus B&C\oplus D&E\oplus F\\
C\oplus F&A\oplus E&B\oplus D\\
E\oplus D&B\oplus F&A\oplus C
\end{pmatrix}$$
with $H=A\oplus\ldots\oplus F$. The associated algebra is $C(G)$, with $G\in\{1,\mathbb Z_2,\mathbb Z_3,S_3\}$.
\end{theorem}

\begin{proof}
We know from Theorem 1.6 that $A_s(3)$ is commutative, and it follows that each of its quotients, and in particular the Hopf image, is commutative as well. 

Now by using Theorem 2.8 we get that our magic basis comes from a magic partition. But the $3 \times 3$ magic partitions are of the following form:
$$\begin{pmatrix}
A\cup B&C\cup D&E\cup F\\
C\cup F&A\cup E&B\cup D\\
E\cup D&B\cup F&A\cup C
\end{pmatrix}$$

This shows that $X$ is of the form in the statement, which proves the result.
\end{proof}

\section{General results}

As explained in the previous section, the study of quantum permutation algebras reduces in principle to that of the magic decompositions of Hilbert spaces.

In this section we present a number of general results, which are essential for this approach. We discuss first the corepresentation theory of Hopf images. 

The tensor powers of a magic unitary $U\in M_n(A)$ are given by:
$$U^{\otimes k}=(U_{i_1j_1}\ldots U_{i_kj_k})_{i_1\ldots i_k,j_1\ldots j_k}$$

In other words, the tensor power is the $n^k\times n^k$ matrix formed by all the length $k$ products between the entries of $U$. Observe that $U^{\otimes k}$ is indeed a magic unitary.

\begin{definition}
Associated to a magic unitary $U\in M_n(A)$ are the spaces
$$Hom(U^{\otimes k},U^{\otimes l})=\{T\in M_{n^l\times n^k}(\mathbb C)| TU^{\otimes k}=U^{\otimes l}T\}$$
with $k,l$ ranging over all positive integers.
\end{definition}

In the case where $U$ is the magic unitary associated to a quantum permutation algebra, we have here Woronowicz's representation theory notions in \cite{wo1}, \cite{wo2}.

The main representation theory problem for a quantum permutation algebra is to compute the above Hom-spaces, for the fundamental magic unitary. The following result from \cite{bb3} reduces this abstract problem to a Hilbert space computation.

\begin{theorem}
Given a representation $\pi:A_s(n)\to B(H)$, we have
$$Hom(u^{\otimes k},u^{\otimes l})=Hom(P^{\otimes k},P^{\otimes l})$$
where $u$ is the fundamental corepresentation of the Hopf image, and $P_{ij}=\pi(u_{ij})$.
\end{theorem}

\begin{proof}
The idea is that the collection of vector spaces on the right forms a tensor category, embedded into the tensor category of finite dimensional Hilbert spaces, and the Hopf image can be shown to be the Tannakian dual of this category, in the sense of \cite{wo2}. We refer to \cite{bb3} for full details regarding this proof.
\end{proof}

As a first application, we will solve now the cocommutative problem. We begin with a technical result, which is of independent interest, in connection with \cite{bi2}.

\begin{proposition}
If a magic decomposition $X$ is non-degenerate, in the sense that $X_{ij}\neq 0$ for any $i,j$, then $Hom(1,u)=\mathbb C$.
\end{proposition}

\begin{proof}
We apply Theorem 3.2, with $k=0$ and $l=1$. We get that for any column vector $T=(t_i)$ we have:
\begin{eqnarray*}
T\in Hom(1,u)
&\iff&T\in Hom(1,P)\\
&\iff&T=PT\\
&\iff&t_i=\sum_jt_jP_{ij},\,\forall i
\end{eqnarray*}

Consider one of the $n$ conditions on the right. The projections $P_{ij}$ are pairwise orthogonal, and by non-degeneracy, they are nonzero. Thus their only linear combinations which are scalars are those having equal coefficients, and we are done.
\end{proof}

A magic partition $(E_{ij})$ is called abelian if the associated group $G\subset S_n$ is abelian.

\begin{theorem}
For a non-degenerate magic decomposition $X_{ij}$, the following are equivalent:
\begin{enumerate}
\item $A$ is cocommutative.

\item $X=X^E$ for an abelian magic partition $E$.
\end{enumerate}
\end{theorem}

\begin{proof}
$(1)\implies (2)$ follows from Proposition 3.3. Indeed, in terms of \cite{bi2}, the condition $Hom(1,u)=\mathbb C$ means that the fundamental coaction of $A$ is ergodic, so it follows from the results in there that if $A$ is cocommutative, then it is commutative. Thus we can apply Theorem 2.8 and Theorem 2.7, and we get the result.

$(2)\implies (1)$ follows from Theorem 2.7. Indeed, we know that in the case $X=X^E$ we have $A=C(G)$. Thus if $E$ is abelian we have $A=C^*(\widehat{G})$, as claimed.
\end{proof}

We discuss now the behavior of the Hopf image with respect to the various product operations at the level of the magic decompositions, or of the magic unitaries. 

The simplest such operation is the tensor product. Given two magic unitaries $U\in M_n(B(H))$ and $V\in M_m(B(K))$, we can form the following matrix:
$$W_{ia,jb}=U_{ij}\otimes V_{ab}$$

It follows from definitions that this matrix is a $nm\times nm$ magic unitary over $B(H\otimes K)$. We call it tensor product of $U,V$, and we use the notation $W=U\otimes V$.

\begin{theorem}
The Hopf algebra associated to $U\otimes V$ is a quotient of $A\otimes B$, where $A$ is the Hopf image for $U$, and $B$ is the Hopf image for $V$.
\end{theorem}

\begin{proof}
The representation of $A_s(nm)$ associated to $U\otimes V$ has a factorization of the following type:
$$\begin{matrix}
A_s(nm)&&\to&&B(H\otimes K)\\
\\
\downarrow&&&&\uparrow\\
\\
A_s(n)\otimes A_s(m)&\to&A\otimes B&\to&B(H)\otimes B(K)
\end{matrix}$$

Indeed, we can define the bottom arrows to be the tensor products of the factorizations associated to $A,B$, and the other arrows to be the canonical ones.

Now since the representation associated to $U\otimes V$ factorizes through $A\otimes B$, we get a morphism as in the statement.
\end{proof}

An interesting generalization of the notion of tensor product, to play a key role in what follows, is the Di\c t\u a product. The following definition is inspired from \cite{di1}. 

\begin{definition}
The Di\c t\u a product of a magic unitary $U\in M_n(B(H))$ with a family of magic unitaries $V^1,\ldots,V^n\in M_m(B(K))$ is the magic unitary given by:
$$W_{ia,jb}=U_{ij}\otimes V^i_{ab}$$
We use the notation $W=U\otimes (V^1,\ldots,V^n)$.
\end{definition}

It follows indeed from definitions that the Di\c t\u a product is a $nm\times nm$ magic unitary over the algebra $B(H\otimes K)$. Observe that in the case where the magic unitaries $V^i$ are all equal, we get an usual tensor product of magic unitaries:
$$U\otimes (V,\ldots,V)=U\otimes V$$

In order to investigate the Hopf images of the Di\c t\u a products, we will need the following definition, which makes us slightly exit from the formalism in \cite{bb3}.

\begin{definition}
The common Hopf image of a family of $C^*$-algebra representations $\pi_i:A_s(n)\to B$ with $i\in I$ is the smallest quantum permutation algebra $A$ realizing a factorization $A_s(n)\to A\to B$ of the representation $\pi_i$, for any $i\in I$.
\end{definition}

As for the usual notion of Hopf image, this construction is best understood in terms of discrete quantum groups. Let $\Gamma$ be the discrete quantum group associated to $A_s(n)$, and let $\Gamma/\Lambda_i$ be the discrete quantum group associated to the Hopf image $A_i$ of the representation $\pi_i$. With these notations, we have the following diagram:
$$\begin{matrix}
A_s(n)&\to&A_i&\to&B\\
\\
||&&||&&||\\
\\
C^*(\Gamma)&\to&C^*(\Gamma/\Lambda_i)&\to&B
\end{matrix}$$

Now if we look for the discrete quantum group associated to the common Hopf image, this must be the quotient of $\Gamma$ by the smallest subgroup containing each $\Lambda_i$. In other words, the common Hopf image is simply given by:
$$A=C^*\left(\Gamma/<\Lambda_i|i\in I>\right)$$

This explanation might seem of course quite heuristic. The idea, however, is that the common Hopf image can be constructed by using a suitable ideal, as in \cite{bb3}.

An alternative approach is simply by using the results in \cite{bb3}: each representation factorizes through its Hopf image $A_s(n)/J_i$, so the common Hopf image should be $A_s(n)/J$, where $J=<J_i>$ is the smallest Hopf ideal containing all the ideals $J_i$.

\begin{theorem}
The algebra associated to $U\otimes (V^1,\ldots,V^n)$ is a quotient of $B*_wA$, where $A$ is the Hopf image for $U$, and $B$ is the common Hopf image for $V^1,\ldots,V^n$.
\end{theorem}

\begin{proof}
Let us first look at the free wreath product between $A_s(m)$ and $A_s(n)$. If we denote by $v,u$ the fundamental corepresentations of these algebras, the product is:
$$A_s(m)*_wA_s(n)=(A_s(m)^{*n}*A_s(n))/<[v_{ab}^{(i)},u_{ij}]=0>$$

It follows from definitions that we can define a map $\Phi:A_s(m)*_wA_s(n)\to B(H\otimes K)$, by mapping the standard generators in the following way:
\begin{eqnarray*}
\Phi(u_{ij})=U_{ij}\otimes 1\\
\Phi(v_{ab}^{(i)})=1\otimes V_{ab}^i
\end{eqnarray*}

We claim now that the representation of $A_s(nm)$ associated to $U\otimes (V^1,\ldots,V^n)$ has a factorization of the following type:
$$\begin{matrix}
A_s(nm)&&\to&&B(H\otimes K)\\
\\
\downarrow&&&&\uparrow\\
\\
A_s(m)*_wA_s(n)&\to&B*_wA&\to&B(H)\otimes B(K)
\end{matrix}$$

Indeed, we can define the bottom arrows to be those coming by factorizing $\Phi$ through the algebra $B*_wA$, and the other arrows to be the canonical ones.

Now since the representation associated to the magic unitary $U\otimes (V^1,\ldots,V^n)$ factorizes through $B*_wA$, we get a morphism as in the statement.
\end{proof}

\section{Magic bases}

We have seen in the previous section that the study of quantum permutation algebras reduces in principle to that of the magic decompositions of Hilbert spaces.

In what follows we restrict attention to the case $H=\mathbb C^n$. It is technically convenient not to choose a basis of $H$, and also to delinearise the 1-dimensional spaces of the magic decomposition, by having as starting point the following definition.

\begin{definition}
A magic basis is a square matrix of vectors $\xi\in M_n(H)$, all whose rows and columns are orthogonal bases of $H$. Associated to $\xi$ are:
\begin{enumerate}
\item The magic unitary matrix given by $P_{ij}$ =  projection on $\xi_{ij}$.

\item The representation $\pi:A_s(n)\to B(H)$ given by $\pi(u_{ij})=P_{ij}$. 

\item The quantum permutation algebra $A=A_\pi$ associated to $\pi$.
\end{enumerate}
\end{definition}

Observe that in case we have such a basis, $H$ is $n$-dimensional, so we have an isomorphism $H\simeq\mathbb C^n$. This isomorphism is not canonical.

The basic example comes from the Latin squares. These are the matrices $\Sigma\in M_n(\mathbb N)$ having the property that all the rows and columns are permutations of $1,\ldots,n$.

We denote by $\Sigma^*$ the Latin square given by $\Sigma^*_{kj}=i$ when $\Sigma_{ij}=k$. Observe that we have $\Sigma^{**}=\Sigma$, and also that we have $\Sigma^{*t}=\Sigma^{t*}$, where $t$ is the transposition.

Here is an example of pair of conjugate Latin squares:
$$\Sigma
=\begin{pmatrix}
1&2&3&4&5\\
3&1&2&5&4\\ 
4&5&1&3&2\\ 
2&4&5&1&3\\ 
5&3&4&2&1
\end{pmatrix}
\quad \quad
\Sigma^*
=\begin{pmatrix}
1&2&3&4&5\\
4&1&2&5&3\\ 
2&5&1&3&4\\ 
3&4&5&1&2\\ 
5&3&4&2&1
\end{pmatrix}$$

If $H$ is a Hilbert space given with an orthogonal basis $b_1,\ldots,b_n$ and $\Sigma\in M_n(\mathbb N)$ is a Latin square, the vectors $\xi_{ij}=b_{\Sigma_{ij}}$ form a magic basis of $H$.

We have the following result, basically proved in \cite{bni}.

\begin{theorem}
For a Latin magic basis $b_\Sigma$ we have $A=C(G)$, where $G\subset S_n$ is the group generated by the rows of $\Sigma^*$.
\end{theorem}

\begin{proof}
It follows from definitions that the magic decomposition associated to $b_\Sigma$ is the magic partition decomposition $X^E$, where $X_k=\mathbb Cb_k$ and $E_{ij}=\{\Sigma_{ij}\}$. Thus we can apply Theorem 2.7, and we get $A=C(G)$, where $G$ is the group associated to $E$.

We know that we have $G=<\sigma_1,\ldots,\sigma_n>$, where $\sigma_k(j)=i$ when $k\in E_{ij}$. Together with $E_{ij}=\{\Sigma_{ij}\}$, this shows that $\sigma_k(j)$ is the unique index $i\in\{1,\ldots,n\}$ such that $\Sigma_{ij}=k$. Thus we have $\sigma_k(j)=\Sigma^*_{kj}$, so $\sigma_k$ is the $k$-th row of $\Sigma^*$, and we are done.
\end{proof}

We call a Latin square $\Sigma$ abelian if the corresponding group $G$ is abelian.

\begin{theorem}
Assume that $\pi:A_s(n)\to M_n(\mathbb C)$ comes from a magic basis.
\begin{enumerate}
\item $A$ is commutative iff $\pi$ comes from a Latin square.

\item $A$ is cocommutative iff $\pi$ comes from an abelian Latin square.
\end{enumerate}
\end{theorem}

\begin{proof}
(1) This follows from Theorem 2.8, because a magic partition decomposition into 1-dimensional subspaces is a Latin square basis.

(2) This follows from Theorem 3.4, because the magic decompositions associated to the magic partitions are non-degenerate.
\end{proof}

We discuss now the corepresentation theory of the Hopf image. 

The Gram graph of a magic basis $(\xi_{ij})$ is defined as follows: the vertices are the pairs of indices $(i,j)$, and there is an edge $(i,l)-(r,j)$ when $<\xi_{lj},\xi_{ir}>\neq 0$.

The following statement is inspired from a result of Jones in \cite{jo2}.

\begin{theorem}
The dimension of $End(u)$ is equal to the number of connected components of the Gram graph of $\xi$. Moreover, this dimension is at most $n$.
\end{theorem}

\begin{proof}
We use Theorem 3.2. For an operator $T=(t_{ij})$, we have:
\begin{eqnarray*}
T\in End(u)
&\iff&T\in End(P)\\
&\iff&\sum_kt_{ik}P_{kj}=\sum_kP_{ik}t_{kj}\\
&\iff&t_{il}\xi_{lj}=\sum_kt_{kj}<\xi_{lj},\xi_{ik}>\xi_{ik}\\
&\iff&t_{il}<\xi_{lj},\xi_{ir}>=t_{rj}<\xi_{lj},\xi_{ir}>\\
&\iff&(t_{il}-t_{rj})<\xi_{lj},\xi_{ir}>=0
\end{eqnarray*}

In terms of the Gram graph, this shows that the condition $T\in End(u)$ is equivalent to the collection of conditions $t_{il}=t_{rj}$, one for each edge $(i,l)-(r,j)$. 

In other words, the entries of $T$ must be constant over the connected components of the Gram graph, and this gives the first result. The second one follows from it.
\end{proof}

For the computation of higher commutants, the idea is to improve Theorem 3.2, by using the following magic basis-specific notions.

\begin{definition}
Associated to a magic basis $\xi_{ij}\in M_n(H)$ are:
\begin{enumerate}
\item The Gram matrix, $G_{ia}^{jb}=<\xi_{ij},\xi_{ab}>$.

\item The higher Gram matrices, $G^k_{i_1\ldots i_k,j_1\ldots j_k}=G_{i_ki_{k-1}}^{j_kj_{k-1}}\ldots G_{i_2i_1}^{j_2j_1}$.
\end{enumerate}
\end{definition}

Observe that we have $G^k\in M_{n^k}(\mathbb C)$. Observe also that $G$ is equal to the first higher Gram matrix, namely $G^2$, but only after a permutation of the indices:
$$G_{ia}^{jb}=G^2_{ai,bj}$$

As a first example, for a basis $\xi=b_\Sigma$ coming from a Latin square, we have:
\begin{eqnarray*}
G_{ia}^{jb}
&=&<\xi_{ij},\xi_{ab}>\\
&=&<b_{\Sigma_{ij}},b_{\Sigma_{ab}}>\\
&=&\delta_{\Sigma_{ij},\Sigma_{ab}}
\end{eqnarray*}

As for the higher Gram matrices, these are given by:
\begin{eqnarray*}
G^k_{ij}
&=&G_{i_ki_{k-1}}^{j_kj_{k-1}}\ldots G_{i_2i_1}^{j_2j_1}\\
&=&\delta(\Sigma_{i_kj_k},\Sigma_{i_{k-1}j_{k-1}})\ldots\delta(\Sigma_{i_2j_2},\Sigma_{i_1j_1})\\
&=&\delta(\Sigma_{i_kj_k},\ldots,\Sigma_{i_1j_1})
\end{eqnarray*}

Here we use generalized Kronecker symbols, for multi-indices. These are by definition given by $\delta(i)=1$ if all the indices of $i$ are equal, and $\delta(i)=0$ if not.

\begin{theorem}
We have the formula
$$Hom(u^{\otimes k},u^{\otimes l})=\{T|T^\circ G^{k+2}=G^{l+2}T^\circ\}$$
where we use the notation $T^\circ=1\otimes T\otimes 1$.
\end{theorem}

\begin{proof}
With the notations in Theorem 3.2, we have the following formula:
$$Hom(u^{\otimes k},u^{\otimes l})=Hom(P^{\otimes k},P^{\otimes l})$$

The vector space on the right consists by definition of the complex $n^l\times n^k$ matrices $T$, satisfying the following relation:
$$TP^{\otimes k}=P^{\otimes l}T$$ 

If we denote this equality by $L=R$, the left term $L$ is given by:
\begin{eqnarray*}
L_{ij}
&=&(TP^{\otimes k})_{ij}\\
&=&\sum_aT_{ia}P^{\otimes k}_{aj}\\
&=&\sum_aT_{ia}P_{a_1j_1}\ldots P_{a_kj_k}
\end{eqnarray*}

As for the right term $R$, this is given by:
\begin{eqnarray*}
R_{ij}
&=&(P^{\otimes l}T)_{ij}\\
&=&\sum_bP^{\otimes l}_{ib}T_{bj}\\
&=&\sum_bP_{i_1b_1}\ldots P_{i_lb_l}T_{bj}
\end{eqnarray*}

Since the elements of $\xi$ span the ambient Hilbert space, the equality $L=R$ is equivalent to the following equality:
$$<L_{ij}\xi_{pq},\xi_{rs}>=<R_{ij}\xi_{pq},\xi_{rs}>$$

In order to compute these quantities, we can use the following well-known formula, expressing a product of rank one projections $P_1,\ldots,P_k$ in terms of the corresponding image vectors $\xi_1,\ldots,\xi_k$:
$$<P_1\ldots P_kx,y>=<x,\xi_k><\xi_k,\xi_{k-1}>\ldots<\xi_2,\xi_1><\xi_1,y>$$

This gives the following formula for $L$:
\begin{eqnarray*}
<L_{ij}\xi_{pq},\xi_{rs}>
&=&\sum_aT_{ia}<P_{a_1j_1}\ldots P_{a_kj_k}\xi_{pq},\xi_{rs}>\\
&=&\sum_aT_{ia}<\xi_{pq},\xi_{a_kj_k}>\ldots<\xi_{a_1j_1},\xi_{rs}>\\
&=&\sum_aT_{ia}G_{pa_k}^{qj_k}G_{a_ka_{k-1}}^{j_kj_{k-1}}\ldots G_{a_2a_1}^{j_2j_1}G_{a_1r}^{j_1s}\\
&=&\sum_aT_{ia}G^{k+2}_{rap,sjq}\\
&=&(T^\circ G^{k+2})_{rip,sjq}
\end{eqnarray*}

As for the right term $R$, this is given by:
\begin{eqnarray*}
<R_{ij}\xi_{pq},\xi_{rs}>
&=&\sum_b<P_{i_1b_1}\ldots P_{i_lb_l}\xi_{pq},\xi_{rs}>T_{bj}\\
&=&\sum_b<\xi_{pq},\xi_{i_lb_l}>\ldots<\xi_{i_1b_1},\xi_{rs}>T_{bj}\\
&=&\sum_bG_{pi_l}^{qb_l}G_{i_li_{l-1}}^{b_lb_{l-1}}\ldots G_{i_2i_1}^{b_2b_1}G_{i_1r}^{b_1s}T_{bj}\\
&=&\sum_bG^{l+2}_{rip,sbq}T_{bj}\\
&=&(G^{l+2}T^\circ)_{rip,sjq}
\end{eqnarray*}

This gives the formula in the statement.
\end{proof}

As a first application, we will solve now the tensor product problem. A tensor product of two magic bases $\xi=\eta\otimes\rho$ is by definition given by $\xi_{ia,jb}=\eta_{ij}\otimes\rho_{ab}$.

\begin{theorem}
The Hopf algebra associated to a tensor product $\xi=\eta\otimes\rho$ is given by $A=B\otimes C$, where $B,C$ are the Hopf algebras associated to $\eta,\rho$.
\end{theorem}

\begin{proof}
We already know from Theorem 3.5 that we have a morphism $B\otimes C\to A$. The point is that, by Tannakian duality, this morphism is injective. Consider indeed the Gram matrices $H,L$ for $\eta,\rho$. Then the Gram matrix of $\xi$ is given by:
\begin{eqnarray*}
G_{ia,IA}^{jb,JB}
&=&<\xi_{ia,jb},\xi_{IA,JB}>\\
&=&<\eta_{ij}\otimes\rho_{ab},\eta_{IJ}\otimes\rho_{AB}>\\
&=&<\eta_{ij},\eta_{IJ}><\rho_{ab},\rho_{AB}>\\
&=&H_{iI}^{jJ}L_{aA}^{bB}
\end{eqnarray*}

Thus the higher Gram matrices of $\xi$ are given by:
\begin{eqnarray*}
G^k_{i_1a_1\ldots i_ka_k,j_1b_1\ldots j_kb_k}
&=&G_{i_ka_k,i_{k-1}a_{k-1}}^{j_kb_k,j_{k-1}b_{k-1}}\ldots G_{i_2a_2,i_1a_1}^{j_2b_2,j_1b_1}\\
&=&H_{i_ki_{k-1}}^{j_kj_{k-1}}L_{a_ka_{k-1}}^{b_kb_{k-1}}\ldots H_{i_2i_1}^{j_2j_1}L_{a_2a_1}^{b_2b_1}\\
&=&H_{i_ki_{k-1}}^{j_kj_{k-1}}\ldots H_{i_2i_1}^{j_2j_1}L_{a_ka_{k-1}}^{b_kb_{k-1}}\ldots L_{a_2a_1}^{b_2b_1}\\
&=&H^k_{i_1\ldots i_k,j_1\ldots j_k}L^k_{a_1\ldots a_k,b_1\ldots b_k}
\end{eqnarray*}

In other words, we have the following equality:
$$G^k=H^k\otimes L^k$$

Now by applying Theorem 4.6, and by using some standard linear algebra indentifications, we get:
\begin{eqnarray*}
End(u^{\otimes k})
&=&\{T|1\otimes T\otimes 1\in (G^{k+2})'\}\\
&=&\{T|1\otimes T\otimes 1\in (H^{k+2})'\otimes (L^{k+2})'\}\\
&=&End((v\otimes w)^{\otimes k})
\end{eqnarray*}

Here $v,w$ are respectively the magic unitary matrices of $B,C$. Now by a standard argument, this equality shows that the morphism $B\otimes C\to A$ is injective on the algebra of coefficients of the even powers of $v\otimes w$. Since we have $1\in v$, $1\in w$, this subalgebra of coefficients is the tensor product itself, and we are done.
\end{proof}

We discuss now the classification problem, for small values of $n$. At $n\leq 3$ it follows from Theorem 2.9 and Theorem 2.10 that the only magic basis is the circular one, and that the corresponding algebra is $C(\mathbb Z_n)$. At $n=4$ we have the following question.

\begin{problem}
What are the magic bases of $\mathbb C^4$, and what are the corresponding Hopf algebras?
\end{problem}

A large class of examples of such magic bases, which altogether provide a faithful representation of the algebra $A_s(4)$, comes from the Pauli matrices. See \cite{bc2}. We don't know if we get in this way all the magic bases at $n=4$.

As for the corresponding Hopf algebras, these are all quotients of $A_s(4)$, so they are subject to the ADE classification result in \cite{bb2}. However, even in the case of the magic bases coming from the Pauli matrices, where some partial results are available \cite{bb3}, \cite{bni}, we don't know exactly how to perform the computation in the general case.

Summarizing, the above problem seems to be of great importance in connection with the previous considerations in \cite{bb2}, \cite{bb3}, \cite{bc2}, \cite{bni}, and its answer would be probably a kind of ultimate result regarding the algebra $A_s(4)$ and its quotients.

\section{Hadamard matrices}

In the reminder of this paper we study the magic bases and the corresponding representations of $A_s(n)$ coming from the complex Hadamard matrices. Most of the preliminary material in this sense can be found as well in the recent paper \cite{bni}.

\begin{definition}
A complex Hadamard matrix is a square matrix $h\in M_n(\mathbb C)$ whose entries are on the unit circle, and whose rows are pairwise orthogonal.
\end{definition}

It follows from definitions that the columns are pairwise orthogonal as well.

These matrices appeared in a paper of Popa, who discovered that a unitary matrix $h\in M_n(\mathbb C)$ is a multiple of a complex Hadamard matrix if and only if the orthogonal MASA condition $\Delta\perp h\Delta h^*$ is satisfied, where $\Delta\subset M_n(\mathbb C)$ is the algebra of diagonal matrices \cite{po1}. Such a pair of orthogonal MASA's produces a commuting square, and the commuting squares are in turn known to classify the finite depth subfactors \cite{po2}.

Due to this fact, the classification problem for the complex Hadamard matrices, and the computation of the corresponding algebraic invariants, quickly became key problems in operator algebras. See Haagerup \cite{ha1}, Jones \cite{jo2} and the book \cite{jsu}. 

For some recent investigations, originating somehow from the same circle of ideas, see Grossman and Jones \cite{gjo}. For a discussion of certain arithmetic aspects, involving arbitrary fields instead of $\mathbb C$, see Bacher, de la Harpe and Jones \cite{bhj}.

The difficulty in the study of complex Hadamard matrices comes from the fact that there is only one basic example, namely the Fourier matrix.

\begin{definition}
The Fourier matrix is $F_n=w^{(i-1)(j-1)}$, where $w=e^{2\pi i/n}$.
\end{definition}

The terminology comes from the fact that $F_n$ is the matrix of the discrete Fourier transform, over the cyclic group $\mathbb Z_n$. We will come back later to this fact, with the remark that the quantum group associated to $F_n$ is indeed $\mathbb Z_n$.

Here are the first three Fourier matrices, with the notation $j=e^{2\pi i/3}$: 
$$F_2=\begin{pmatrix}
1&1\\
1&-1
\end{pmatrix}\quad\quad 
F_3=\begin{pmatrix}
1&1&1\\ 
1&j&j^2\\
1&j^2&j
\end{pmatrix}\quad\quad 
F_4=\begin{pmatrix}
1&1&1&1\\ 
1&i&-1&-i\\
1&-1&1&-1\\
1&-i&-1&i
\end{pmatrix}$$

Observe that $F_n$ has the property that its first row and column consist only of $1$'s. This is due to the exponent $(i-1)(j-1)$ instead of $ij$, in the above definition.

This normalization can be in fact always done, up to equivalence.

\begin{definition}
Let $h,k$ be two complex Hadamard matrices.
\begin{enumerate}
\item $h$ is called dephased if its first row and column consist only of $1$'s.

\item $h,k$ are called equivalent if one can pass from one to the other by permuting the rows or columns, or by multiplying them by complex numbers of modulus $1$.
\end{enumerate}
\end{definition}

Observe that any complex Hadamard matrix can be supposed to be in dephased form, up to the above equivalence relation. With a few exceptions, we will do so.

Note that we do not include the transposition in the above operations. This is because at the level of associated Hopf algebras, the transposition corresponds to a highly non-trivial operation, making correspond for instance algebras of type $A*_wB$ to algebras of type $B*_wA$. See section 11 below for a concrete such example.

One can prove that at $n=2,3$ the Fourier matrix is the only complex Hadamard matrix, modulo equivalence. At $n=4$ we have the following general example, depending on a complex parameter on the unit circle, $|q|=1$:
$$F_{22}^q
=\begin{pmatrix}
1&1&1&1\\
1&q&-1&-q\\
1&-1&1&-1\\ 
1&-q&-1&q
\end{pmatrix}$$

The notation comes from the fact that at $q=1$ we get a matrix which is equivalent to $F_2\otimes F_2$. Observe also that at $q=i$ we get a matrix which is equivalent to $F_4$. 

At $n=5$ we have the Fourier matrix, based on the root of unity $w=e^{2\pi i/5}$: 
$$F_5=\begin{pmatrix} 
1&1&1&1&1\\  
1&w&w^2&w^3&w^4\\  
1&w^2&w^4&w&w^3\\ 
1&w^3&w&w^4&w^2\\ 
1&w^4&w^3&w^2&w
\end{pmatrix}$$

The following remarkable result is due to Haagerup \cite{ha1}.

\begin{theorem}
At $n=2,3,4,5$ the above matrices $F_2,F_3,F_{22}^q,F_5$ are the only complex Hadamard matrices, modulo equivalence.
\end{theorem}

At $n=6$ the situation is much more complicated. First, we have the Fourier matrix, based on the root of unity $w=-j^2$, where $j=e^{2\pi i/3}$:
$$F_6=\begin{pmatrix}
1&1&1&1&1&1\\
1&-j^2&j&-1&j^2&-j\\
1&j&j^2&1&j&j^2\\ 
1&-1&1&-1&1&-1\\
1&j^2&j&1&j^2&j\\
1&-j&j^2&-1&j&-j^2
\end{pmatrix}$$

As it was the case with $F_4$, this matrix can be deformed, with the space of parameters consisting this time of twice the product of the unit circle with itself. This deformation appears as particular case of a quite general construction, to be discussed later on.

A first matrix which is not equivalent to $F_6$, nor to its deformations, is the Tao matrix \cite{tao}, based on the root of unity $j=e^{2\pi i/3}$:
$$T=\begin{pmatrix}
1&1&1&1&1&1\\ 
1&1&j&j&j^2&j^2\\ 
1&j&1&j^2&j^2&j\\
1&j&j^2&1&j&j^2\\ 
1&j^2&j^2&j&1&j\\ 
1&j^2&j&j^2&j&1
\end{pmatrix}$$

Another remarkable example, this time depending on a complex parameter $|q|=1$, is the following matrix, constructed in \cite{ha1} at $q=1$, and in \cite{di1} for any $|q|=1$:
$$H^q=\begin{pmatrix}
1&1&1&1&1&1\\
1&-1&i&i&-i&-i\\ 
1&i&-1&-i&q&-q\\ 
1&i&-i&-1&-q&q\\
1&-i&\bar{q}&-\bar{q}&i&-1\\ 
1&-i&-\bar{q}&\bar{q}&-1&i
\end{pmatrix}$$

Yet another example, this time with circulant structure, is the Bj\"orck-Fr\"oberg matrix \cite{bfr}, built by using one of the two roots of $a^2-(1-\sqrt{3})a+1=0$:
$$BF=\begin{pmatrix}
1&ia&-a&-i&-\bar{a}&i\bar{a}\\
i\bar{a}&1&ia&-a&-i&-\bar{a}\\
-\bar{a}&i\bar{a}&1&ia&-a&-i\\
-i&-\bar{a}&i\bar{a}&1&ia&-a\\
-a&-i&-\bar{a}&i\bar{a}&1&ia\\
ia&-a&-i&-\bar{a}&i\bar{a}&1
\end{pmatrix}$$

The classification problem is open at $n=6$, where a certain number of results are available \cite{bb+}, \cite{msz}, \cite{szo}. The main result so far concerns the self-adjoint case \cite{ben}.

At $n=7$ we have the following matrix, discovered by Petrescu \cite{pet}:
$$P^q
=\begin{pmatrix}
1&1&1&1&1&1&1\\
1&qw&qw^4&w^5&w^3&w^3&w\\
1&qw^4&qw&w^3&w^5&w^3&w\\
1&w^5&w^3&\bar{q}w&\bar{q}w^4&w&w^3\\
1&w^3&w^5&\bar{q}w^4&\bar{q}w&w&w^3\\
1&w^3&w^3&w&w&w^4&w^5\\
1&w&w&w^3&w^3&w^5&w^4
\end{pmatrix}$$

Here $w=e^{2\pi i/6}$. This matrix, a non-trivial deformation of prime order, was found by using a computer program, and came as a big surprise at the time of \cite{pet}.

At $n=7$, or bigger, very less seems to be known. A number of abstract or concrete results here are available from \cite{bur}, \cite{di2}, \cite{ha2}, \cite{mrs}, \cite{nic}, \cite{tz1}, \cite{tz2}.

\section{Symmetry algebras}

We will associate now a quantum permutation algebra to any complex Hadamard matrix. Let $h\in M_n(\mathbb C)$ be such a matrix, and denote its rows by $h_1,\ldots,h_n$. The entries of $h$ being elements on the unit cercle, they are invertible. Thus $h_1,\ldots,h_n$ can be regarded as being invertible elements of the algebra $\mathbb C^n$.

\begin{proposition}
The vectors $\xi_{ij}=h_i/h_j$ form a magic basis of $\mathbb C^n$.
\end{proposition}

\begin{proof}
The Hadamard condition tells us that the scalar products between the rows of $h$ are given by $<h_i,h_j>=n\,\delta_{ij}$. Thus the scalar product between two vectors on the same column of $\xi$ is given by:
\begin{eqnarray*}
<\xi_{ij},\xi_{kj}>
&=&<h_i/h_j,h_k/h_j>\\
&=&n<h_i,h_k>\\
&=&n^2\,\delta_{ik}
\end{eqnarray*}

A similar computation works for the rows, and we are done.
\end{proof}

We can therefore apply the general constructions in section 4. It is convenient to write down the definition of all objects involved.

\begin{definition}
Let $h\in M_n(\mathbb C)$ be a complex Hadamard matrix.
\begin{enumerate}
\item $h_1,\ldots,h_n$ are the rows of $h$, regarded as elements of $\mathbb C^n$.

\item $\xi$ is the magic basis of $\mathbb C^n$ given by $\xi_{ij}=h_i/h_j$.

\item $P_{ij}$ is the orthogonal projection on $\xi_{ij}$.

\item $\pi:A_s(n)\to B(H)$ is the representation given by $\pi(u_{ij})=P_{ij}$. 

\item $A$ is the quantum permutation algebra associated to $\pi$.
\end{enumerate}
\end{definition}

As explained in the introduction, this construction has been known for some time, but the whole subject is quite slowly evolving. The idea is that the quantum permutation group $G$ associated to the algebra $A$ encodes the ``quantum symmetries'' of $h$, and the hope would be that the quantum permutation groups could be used in order to approach the main problems regarding the complex Hadamard matrices.

We begin our study by carefully reviewing the material in \cite{bni}, by using the abstract machinery developed in the previous sections.

\begin{proposition}
The construction $h\to A$ has the following properties:
\begin{enumerate}
\item For the Fourier matrix $F_n$ we have $A=C(\mathbb Z_n)$.

\item For a tensor product $h=h'\otimes h''$ we have $A=A'\otimes A''$.
\end{enumerate}
\end{proposition}

\begin{proof}
(1) The Fourier matrix is formed by the powers of the root of unity $w=e^{2\pi i/n}$. In terms of the vector $\rho=(1,w,\ldots,w^{n-1})$, the rows of $h=F_n$ are the given by $h_i=\rho^{i-1}$, so the corresponding magic basis is given by $\xi_{ij}=\rho^{i-j}$. But this is a Latin magic basis, and by applying Theorem 4.2 we get the result.

(2) It follows from definitions that at the level of associated magic bases we have $\xi=\xi'\otimes\xi''$, so by applying Theorem 4.7 we get the result.
\end{proof}

As a consequence of the above two results, for a tensor product of Fourier matrices, the corresponding quantum permutation algebra $A$ is commutative. As pointed out in \cite{bni}, the converse holds, and in fact, we have the following general result.

\begin{theorem}
For an Hadamard matrix, the following are equivalent:
\begin{enumerate}
\item $A$ is commutative.

\item $A$ is cocommutative.

\item $A\simeq C(\mathbb Z_{n_1}\times\ldots\times\mathbb Z_{n_k})$, for some numbers $n_1,\ldots,n_k$.

\item $h\simeq F_{n_1}\otimes\ldots\otimes F_{n_k}$, for some numbers $n_1,\ldots,n_k$.
\end{enumerate}
\end{theorem}

\begin{proof}
$(1)\implies (4)$ follows from Theorem 4.3. Indeed, if $A$ is commutative then the corresponding magic basis must come from a Latin square, and a direct computation, performed in \cite{bni}, shows that $F$ must be a tensor product of Fourier matrices.

$(4)\implies (3)$ follows from the above two results.

$(3)\implies (2)$ is clear.

$(2)\implies (1)$ follows from Theorem 4.3.
\end{proof}

We discuss now the computation of the Hom-spaces for the fundamental corepresentation. The following result has been basically known since \cite{ba1}. In its subfactor or planar algebra version, the result has been known for a long time, see \cite{jo2}, \cite{jsu}.

\begin{theorem}
We have $T\in Hom(u^{\otimes k},u^{\otimes l})$ if and only if $T^\circ G^{k+2}=G^{l+2}T^\circ$, where:
\begin{enumerate}
\item $T^\circ=id\otimes T\otimes id$.

\item $G_{ia}^{jb}=\sum_{k=1}^nh_{ik}\bar{h}_{jk}\bar{h}_{ak}h_{bk}$.

\item $G^k_{i_1\ldots i_k,j_1\ldots j_k}=G_{i_ki_{k-1}}^{j_kj_{k-1}}\ldots G_{i_2i_1}^{j_2j_1}$.
\end{enumerate}
\end{theorem}

\begin{proof}
This follows indeed from Theorem 4.6. For a basis $\xi_{ij}=h_i/h_j$ coming from an Hadamard matrix, we have:
\begin{eqnarray*}
G_{ia}^{jb}
&=&<\xi_{ij},\xi_{ab}>\\
&=&<h_i/h_j,h_a/h_b>\\
&=&<(h_{ik}/h_{jk})_k,(h_{ak}/h_{bk})_k>\\
&=&\sum_{k=1}^nh_{ik}\bar{h}_{jk}\bar{h}_{ak}h_{bk}
\end{eqnarray*}

This gives the result.
\end{proof}

We discuss now the various product operations for complex Hadamard matrices. Observe that the tensor product problem has already been solved.

The following product operations, the first one due to Di\c t\u a \cite{di1}, and the second one being inspired from it, will play a key role in what follows.

\begin{definition}
We have the following product operations:
\begin{enumerate}
\item The Di\c t\u a product of an Hadamard matrix $h\in M_n(\mathbb C)$ with a family of Hadamard matrices $k^1,\ldots,k^n\in M_m(\mathbb C)$ is $h\otimes (k^1,\ldots,k^n)=(h_{ij}k^j_{ab})_{ia,jb}$.

\item The Di\c t\u a deformation of a tensor product $h\otimes k\in M_{nm}(\mathbb C)$, with matrix of parameters $l\in M_{m\times n}(\mathbb T)$, is $h\otimes_lk=(h_{ij}l_{aj}k_{ab})_{ia,jb}$.
\end{enumerate}
\end{definition}

The above operations are both given in a compact form, by using some standard tensor product identifications. For practical purposes, however, the usual matrix notation is more convenient. In matrix notation, the Di\c t\u a product is given by:
$$h\otimes(k^1,\ldots,k^n)
=\begin{pmatrix}
h_{11}k^1&\ldots&h_{1n}k^n\\
\ldots&\ldots&\ldots\\
h_{n1}k^1&\ldots&h_{nn}k^n
\end{pmatrix}$$

As for the Di\c t\u a deformation, this is by definition the following Di\c t\u a product:
$$h\otimes_lk=h\otimes
\left( \begin{pmatrix}
l_{11}k_{11}&\ldots&l_{11}k_{1m}\\
\ldots&\ldots&\ldots\\
l_{m1}k_{m1}&\ldots&l_{m1}k_{mm}
\end{pmatrix},\ldots,
\begin{pmatrix}
l_{1n}k_{11}&\ldots&l_{1n}k_{1m}\\
\ldots&\ldots&\ldots\\
l_{mn}k_{m1}&\ldots&l_{mn}k_{mm}
\end{pmatrix}\right)$$

It is possible of course to further expand the Di\c t\u a product, see section 10 below.

Observe that these notions generalize the usual tensor product, because $h\otimes k$ is equal to $h\otimes_{\mathbb I}k=h\otimes(k,\ldots,k)$, where $\mathbb I$ is the matrix filled with 1's. 

The Di\c t\u a product can be, however, a quite complicated construction. 

\begin{proposition}
$F_{22}^q$ is a Di\c t\u a deformation of $F_2\otimes F_2$.
\end{proposition}

\begin{proof}
Consider indeed the following Di\c t\u a deformation:
$$h^q=\begin{pmatrix}1&1\\ 1&-1\end{pmatrix}
\otimes_{\begin{pmatrix}1&1\\ 1&q\end{pmatrix}}
\begin{pmatrix}1&1\\ 1&-1\end{pmatrix}$$

In Di\c t\u a product notation, this matrix is given by:
$$h^q=\begin{pmatrix}1&1\\ 1&-1\end{pmatrix}
\otimes\left( 
\begin{pmatrix}1&1\\ 1&-1\end{pmatrix},
\begin{pmatrix}1&1\\ q&-q\end{pmatrix}
\right)$$

Thus we have the following formula:
$$h^q=\begin{pmatrix}
1&1&1&1\\
1&-1&q&-q\\
1&1&-1&-1\\ 
1&-1&-q&q
\end{pmatrix}$$

The matrix on the right being equivalent to $F_{22}^q$, this gives the result.
\end{proof}

Observe that in the above example, the first row and column of the parameter matrix $l$ consist only of 1's. This normalization can be made as well in the general case.

The following result should be related to the considerations in \cite{cni}.

\begin{theorem}
We have the following results:
\begin{enumerate}
\item The algebra associated to $h\otimes (k^1,\ldots,k^n)$ is a quotient of $B*_wA$, where $A$ is the algebra associated to $h$, and $B$ is the algebra associated to $k^1,\ldots,k^n$.

\item The algebra associated to $h\otimes_lk$ is a quotient of $B*_wA$, where $A$ is the algebra associated to $h$, and $B$ is the algebra associated to $k$.
\end{enumerate}
\end{theorem}

\begin{proof}
This follows from Theorem 3.8, due to the compatibility between the Di\c t\u a products of Hadamard matrices, and of magic unitaries.
\end{proof}

\begin{problem}
For which Di\c t\u a deformations is the associated algebra isomorphic to the ambient free wreath product?
\end{problem}

We believe that this happens for instance when the matrix of parameters $l$ is generic. Here by ``generic'' we mean for instance having the entries algebrically independent over $\mathbb Q$, but some weaker conditions are actually expected to be sufficient. 

This conjecture is verified for $h=k=F_2$, thanks to the computations in \cite{bni}.

The natural idea for verifying the conjecture would be via Tannakian duality, but the Tannakian description of the free wreath products is not available yet. So far we have only a conjecture in this sense, regarding the dimensions of the Hom-spaces \cite{bb1}.

\section{Butson matrices}

Most of the examples of Hadamard matrices given in section 5 are based on certain roots of unity. We have here the following definition.

\begin{definition}
The level of a complex Hadamard matrix $h\in M_n(\mathbb C)$ is the smallest number $l\in\{1,2,\ldots,\infty\}$ such that all the entries of $h$ are $l$-th roots of unity.
\end{definition}

Here we agree that a root of unity of infinite order is simply a number on the unit circle. The level of a complex Hadamard matrix $h$ will be denoted $l(h)$.

The matrices having level $l<\infty$ were first investigated by Butson in \cite{but}. In this section we discuss the main combinatorial problems regarding such matrices.

\begin{definition}
The Butson class $H_n(l)$ consists of Hadamard matrices in $M_n(\mathbb C)$ having as entries the $l$-th roots of unity. In particular:
\begin{enumerate}
\item $H_n(2)$ is the set of all $n\times n$ real Hadamard matrices.

\item $H_n(l)$ is the set of $n\times n$ Hadamard matrices of level $l'|l$.

\item $H_n(\infty)$ is the set of all $n\times n$ Hadamard matrices.
\end{enumerate}
\end{definition}

The basic problem regarding the Butson matrices, that is related as well to the present Hopf algebra considerations, is the characterization of the pairs $(n,l)$ such that $H_n(l)\neq 0$. We have here the following fundamental result, due to Sylvester \cite{syl}.

\begin{theorem}
If $H_n(2)\neq\emptyset$ then $n=2$ or $4|n$.
\end{theorem}

\begin{proof}
Let $h\in H_n(2)$, with $n\geq 3$. By using the equivalence relation, we may assume that the first three rows have a normalized block decomposition, as follows:  
$$h=\begin{pmatrix}
1&1&1&1\\
1&1&-1&-1\\
1&-1&1&-1\\
\ldots&\ldots&\ldots&\ldots
\end{pmatrix}$$

Now let $a,b,c,d$ be the lengths of the blocks in the third row. The orthogonality relations between the first three rows give $a+b=c+d$, $a+c=b+d$ and $a+d=b+c$, so we have $a=b=c=d$, and we can conclude that we have $4|n$.
\end{proof}

The Hadamard conjecture, named after \cite{had}, states that the converse of the above result is true: if $4|n$ then $H_n(2)\neq\emptyset$. This question is reputed to be of remarkable difficulty, and the numeric verification so far goes up to $n=664$. See \cite{kta}, \cite{lgo}.

For general exponents $l>2$, the formulation of such conjectures is a quite delicate problem, because there are many obstructions on $(n,l)$, of quite different nature. 

The basic result here, coming from the results of Lam and Leung in \cite{lle}, is as follows: 

\begin{theorem}
If $H_n(l)\neq\emptyset$ and $l=p_1^{a_1}\ldots p_s^{a_s}$ then $n\in p_1\mathbb N+\ldots+p_s\mathbb N$.
\end{theorem}

\begin{proof}
The simplest particular case of this statement is the condition ``$l=2$ implies $2|n$'', weaker than the Sylvester obstruction, and whose proof is elementary. As pointed out by Butson in \cite{but}, a similar argument applies to the general case where $l=p$ is prime. Moreover, as observed by Winterhof in \cite{win}, the case $l=p^a$ is similar.

In the general case, the idea is the same: the obstruction  comes from the orthogonality of the first two rows. Indeed, this orthogonality condition tells us that in order to have $H_n(l)\neq 0$, the number $n$ must belong to the following set:
$$\Lambda_l=\left\{n\in\mathbb N\ \Big|\ \exists\,w_1,\ldots,w_n,\,w_i^l=1,\,\sum w_i=0\right\}$$

For $p$ prime, we call $p$-cycle the formal sum of all roots of unity of order $p$, that might be globally rotated, i.e. multiplied by a complex number of modulus 1. Since the actual sum of a cycle is 0, we have $p_1,\ldots,p_s\in\Lambda_l$, so we get:
$$p_1\mathbb N+\ldots+p_s\mathbb N\subset\Lambda_l$$

The point is that, by the general results of Lam and Leung in \cite{lle}, this inclusion is an equality. Thus the condition $n\in\Lambda_l$ is in fact the one in the statement. 
\end{proof}

In order to get more insight into the structure of Butson matrices, we have to understand the precise meaning of the Lam-Leung result. The situation is as follows:

At $s=1,2$ this follows from a finer result, stating that any vanishing sum of $l$-roots of unity can be decomposed into cycles. The proof of this latter result is elementary at $s=1$, and follows from a routine computation at $s=2$.

At $s=3$ the situation becomes considerably more complicated, because there exist vanishing sums which don't decompose into cycles. The idea is that given any three prime numbers $p,q,r$, we can produce a ``non-trivial'' vanishing sum by substracting a $p$-cycle from a suitable union of $q$-cycles and $r$-cycles. 

Here is the simplest example of such a sum, with $w=e^{2\pi i/30}$:
$$S=w^5+w^6+w^{12}+w^{18}+w^{24}+w^{25}$$

The fact that $S$ vanishes indeed can be checked as follows:
\begin{eqnarray*}
S&=&(w^6+w^{12}+w^{18}+w^{24})+(w^5+w^{25})\\
&=&(w^0+w^6+w^{12}+w^{18}+w^{24})+(w^5+w^{15}+w^{25})-(w^0+w^{15})\\
&=&0
\end{eqnarray*}

However, by drawing the elements of $S$ on the unit circle, we can see that $S$ cannot decompose as a sum of cycles. Observe however that the length of this ``non-trivial'' vanishing sum is $6\in 2\mathbb N+3\mathbb N+5\mathbb N$, as predicted by the general results in \cite{lle}. 

As a conclusion,  the following happens: ``a vanishing sum of roots of unity has the same length as a sum of cycles, althought it isn't necessarily a sum of cycles''.

These considerations suggest the following definition.

\begin{definition}
A Butson matrix is called regular if the scalar product of each pair of rows decomposes as a sum of cycles.
\end{definition}

In other words, associated to a given matrix $h\in H_n(l)$ are the $n(n-1)/2$ relations stating that the rows are pairwise orthogonal. Each of these relations is a vanishing sum of $l$-roots of unity, and the regularity condition is that each of these vanishing sums decomposes as a sum of $p$-cycles, with $p$ ranging over the prime divisors of $l$.

The point is that all the known examples of Butson matrices seem to be regular. For instance for the Petrescu matrix $P^q$, each vanishing sum coming from the orthogonality of the rows consists of two 2-cycles and a 3-cycle.

\begin{conjecture}
The regularity condition is automatic.
\end{conjecture}

This conjecture is of particular interest in connection with the Lam-Laung obstruction, because for a regular matrix, the obstruction is trivially satisfied. In other words, this conjecture would provide a substantial extension of the Lam-Laung obstruction.

Observe that, according the considerations preceding Definition 7.5, the conjecture holds for any $h\in H_n(l)$, with $l$ having at most 2 prime factors. However, once again by the above considerations, a new idea, which must be Hadamard matrix-specific, would be needed for exponents $l$ having at least 3 prime factors.

We discuss now some other obstructions on $(n,l)$. A basic obstruction, coming this time from all the rows, is the following one, due to de Launey \cite{lau}: 

\begin{theorem}
If $H_n(l)\neq\emptyset$ then there is $d\in\mathbb Z[e^{2\pi i/l}]$ such that $|d|^2=n^n$.
\end{theorem}

\begin{proof}
This follows from $hh^*=nI_n$, by applying the determinant: indeed, we get $|{\rm det}(h)|^2=n^n$. The corresponding obstructions on $(l,n)$ are of quite subtle arithmetic nature, the simplest consequence being ``$l=6$ implies $n\neq 5$''. See de Launey \cite{lau}.
\end{proof}

Finally, we have the following obstruction, due to Haagerup \cite{ha1}:

\begin{theorem}
If $H_5(l)\neq\emptyset$ then $5|l$.
\end{theorem}

\begin{proof}
This follows from Haagerup's classification results in \cite{ha1}. Indeed, since the Fourier matrix $F_5$ is the only complex Hadamard matrix at $n=5$, up to equivalence, each matrix $h\in H_5(l)$ must be obtained from it by permuting the rows and the columns, or by multiplying them by certain roots of unity. In terms of levels, this gives $l(F_5)|l(h)$, and from $l(F_5)=5$ and $l(h)|l$ we get the result.
\end{proof}

We would like to present as well the following original result, that we found by carefully looking at the proof of the Sylvester obstruction.

\begin{theorem}
Assume $H_n(l)\neq\emptyset$.
\begin{enumerate}
\item If $n=p+2$ with $p\geq 3$ prime, then $l\neq 2p^b$.

\item If $n=2q$ with $p>q\geq 3$ primes, then $l\neq 2^ap^b$.
\end{enumerate}
\end{theorem}

\begin{proof}
We use the logarithmic writing for the elements of $H_n(l)$, with numbers $k\in\{0,1,\ldots, l-1\}$ standing for the corresponding roots of unity $e^{2k\pi i/l}$. Assume that a matrix $h$ contradicting the statement exists, and write it in logarithmic form.

(1) We know that each row of $h$ contains one $2$-cycle and one $p$-cycle. The two elements of the $2$-cycle have opposite parities, while the elements of the $p$-cycle have the same parity. Therefore, each row of $h$ has either exactly one odd entry or exactly one even entry. Moreover, the same applies to the difference between rows, since rows correspond to pairwise orthogonal vectors.

Let $L_1,L_2$ be two rows of $h$. We have 3 cases:

Case 1. If $L_1$ and $L_2$ both have exactly one even entry, then $L_2-L_1$ has either no odd entry, if the even entries of $L_1$ and $L_2$ are at the same position, or exactly two odd entries, if these even entries are at different positions.

Case 2. The same holds if $L_1$ and $L_2$ both have exactly one odd entry.

Case 3. If $L_1$ has exactly one even entry, and $L_2$ has exactly one odd entry, then  $L_2-L_1$ has either no even entry, if the positions correspond, or exactly two even entries, if the positions are different.

We can see that in all the three cases, $L_2-L_1$ cannot have either exactly one odd entry or exactly one even entry, a contradiction.

(2) We know that each row of $h$ is a union of $2$-cycles and of $p$-cycles. Since $p>q$, there can be no $p$-cycle, since one $p$-cycle would leave an odd number of elements which cannot be grouped in $2$-cycles. So, each row of $h$ is a union of $2$-cycles.

The same argument shows that the difference between two rows is also a union of $2$-cycles. Thus the reduction of $h$ modulo 2 is a real Hadamard matrix, so the usual Sylvester obstruction applies, and shows that there is no such matrix, since $q$ is odd.
\end{proof}

We are now in position of evaluating the ``strength'' of our set of obstructions. The relevant quantity here is the pair $(N,L)$ such that ``for any $n\leq N,l\leq L$, either $H_n(l)\neq\emptyset$ due to an explicit example, or $H_n(l)=\emptyset$ due to one of the obstructions''. Here the pair $(N,L)$ is chosen as for $N+L$ to be maximal, and by using maximality with respect to the lexicographic order, in the case of ambiguity.

With the above set of obstructions we have $(N,L)=(10,14)$, and the result is best stated as follows.

\begin{theorem}
For any $n\leq 10$ and $l\leq 14$, one of the following happens:
\begin{enumerate}
\item Either $H_n(l)\neq\emptyset$, due to an explicit example $X_n^l\in H_n(l)$.

\item Or $H_n(l)=\emptyset$, due to one of the above obstructions.
\end{enumerate}
\end{theorem}

\begin{proof}
We use the following notations for the various known obstructions:
\begin{enumerate}
\item $\circ$ denotes the Lam-Leung obstruction (Theorem 7.4).

\item $\circ_l$ denotes the de Launey obstruction (Theorem 7.7).

\item $\circ_h$ denotes the Haagerup obstruction (Theorem 7.8).

\item $\circ_s$ denotes the Sylvester obstructions (Theorems 7.3 and 7.9).
\end{enumerate}

Also, we denote by $H,P$ the Haagerup and Petrescu matrices, taken at $q=1$, and for $k_1,\ldots,k_s\in\{2,3\}$ we use the notation $F_{k_1\ldots k_s}=F_{k_1}\otimes\ldots\otimes F_{k_s}$. 

We claim that we have the following table, describing for each $n,l$ as in the statement, either an explicit matrix in $H_n(l)$, or an obstruction which applies to $(n,l)$:
\setlength{\extrarowheight}{2pt}
{\small\begin{center}
\begin{tabular}[t]{|l|l|l|l|l|l|l|l|l|l|l|l|l|l|l|l|}
\hline $n\backslash l$\!
&2&3&4&5&6&7&8&9&10&11&12&13&14\\
\hline 2&$F_2$&$\circ$&$F_2$&$\circ$&$F_2$&$\circ$&$F_2$&$\circ$&$F_2$&$\circ$&$F_2$&$\circ$&$F_2$\\
\hline 3&$\circ$&$F_3$&$\circ$&$\circ$&$F_3$&$\circ$&$\circ$&$F_3$&$\circ$&$\circ$&$F_3$&$\circ$&$\circ$\\
\hline 4&$F_{22}$&$\circ$&$F_{22}$&$\circ$&$F_{22}$&$\circ$&$F_{22}$&$\circ$&$F_{22}$&$\circ$&$F_{22}$&$\circ$&$F_{22}$\\
\hline 5&$\circ$&$\circ$&$\circ$&$F_5$&$\circ_l$&$\circ$&$\circ$&$\circ$&$F_5$&$\circ$& $\circ_h$&$\circ$&$\circ$\\
\hline 6&$\circ_s$&$T$&$H$&$\circ$&$T$&$\circ$&$H$&$T$&$\circ_{s}$&$\circ$&$T$&$\circ$&$\circ_{s}$\\
\hline 7&$\circ$&$\circ$&$\circ$&$\circ$&$P$&$F_7$&$\circ$&$\circ$&$\circ_{s}$&$\circ$&$P$&$\circ$&$F_7$\\
\hline 8&$F_{222}\!$&$\circ$&$F_{222}\!$&$\circ$&$F_{222}\!$&$\circ$&$F_{222}\!$&$\circ$&$F_{222}\!$&$\circ$&$F_{222}\!$&$\circ$&$F_{222}\!$\\
\hline 9&$\circ$&$F_{33}$&$\circ$&$\circ$&$F_{33}$&$\circ$&$\circ$&$F_{33}$&$X_9^{10}$&$\circ$&$F_{33}$&$\circ$&$\circ_{s}$\\
\hline 10\!&$\circ_s$&$\circ$&$X^4_{10}$&$X^5_{10}$&$X^6_{10}$&$\circ$&$X^4_{10}$&$\circ$&$F_{10}$&$\circ$&$X^4_{10}$&$\circ$&$\circ$\\
\hline
\end{tabular}
\end{center}}
\setlength{\extrarowheight}{0pt}
\bigskip

Indeed, the missing matrices can be chosen, in logarithmic notation, as follows:
{\small 
$$\hskip2mm
X^{10}_9 =
\left(
\begin{array}{ccccccccc}
0&0&0&0&0&0&0&0&0\\
0&5&3&3&5&9&8&7&1\\
0&4&5&7&1&3&5&9&9\\
0&3&7&5&1&8&9&3&5\\
0&9&1&5&5&3&7&2&7\\
0&9&5&1&3&5&1&7&6\\
0&1&7&9&6&1&5&5&3\\
0&7&9&4&9&5&3&5&1\\
0&5&2&9&7&7&3&1&5
\end{array}
\right)
\hskip8mm 
X^4_{10} =
\left(
\begin{array}{cccccccccc}
0&0&0&0&0&0&0&0&0&0\\
0&2&3&3&3&3&1&1&1&1\\
0&3&2&1&1&3&3&3&1&1\\
0&3&1&2&3&1&3&1&3&1\\
0&3&1&3&2&1&1&3&1&3\\
0&3&3&1&1&2&1&1&3&3\\
0&1&3&3&1&1&2&3&3&1\\
0&1&3&1&3&1&3&2&1&3\\
0&1&1&3&1&3&3&1&2&3\\
0&1&1&1&3&3&1&3&3&2
\end{array}
\right) $$
$$X^5_{10}=
\left(\begin{array}{cccccccccccccc}
0&0&0&0&0&0&0&0&0&0\\
0&0&1&1&2&2&3&3&4&4\\
0&1&0&3&2&4&1&4&2&3\\
0&1&3&4&3&1&0&2&4&2\\
0&2&3&0&1&3&4&1&2&4\\
0&2&4&2&0&1&3&4&3&1\\
0&3&1&2&4&0&4&2&1&3\\
0&3&2&4&1&4&2&3&0&1\\
0&4&2&1&4&3&1&0&3&2\\
0&4&4&3&3&2&2&1&1&0
\end{array}\right)
\hskip5mm
X^6_{10}=
\left(\begin{array}{cccccccccccccc}
0&0&0&0&0&0&0&0&0&0\\
0&4&1&5&3&1&3&3&5&1\\
0&1&2&3&5&5&1&3&5&3\\
0&5&3&2&1&5&3&5&3&1\\
0&3&5&1&4&1&1&5&3&3\\
0&3&3&3&3&3&0&0&0&0\\
0&1&1&5&3&4&3&0&2&4\\
0&1&5&3&5&2&4&3&2&0\\
0&5&3&5&1&2&0&2&3&4\\
0&3&5&1&1&4&4&2&0&3
\end{array}\right)$$}

This justifies the above table, and we are done.
\end{proof}

We don't know what happens at $n\leq 10$ and $l=15$, nor about what happens at $n=11$ and $l\leq 14$. In each of these two cases, after applying the obstructions, remembering the known examples, and constructing some more examples by using our home software, one case of the extended table is left blank.

\section{The Tao matrix}

Thanks to Haagerup's classification result in \cite{ha1}, all the complex Hadamard matrices are known at $n\leq 5$. As explained in section 5, at $n=6$ the general classification of complex Hadamard matrices looks like a difficult task. See \cite{ben}, \cite{bb+}, \cite{msz}, \cite{szo}. 

The point, however, is that the matrices in the Butson class can be fully classified at $n=6$. This will be basically our goal for this section, and for the next two ones.

In this section we find an abstract characterization of the Tao matrix:
$$T=\begin{pmatrix}
1&1&1&1&1&1\\ 
1&1&j&j&j^2&j^2\\ 
1&j&1&j^2&j^2&j\\
1&j&j^2&1&j&j^2\\ 
1&j^2&j^2&j&1&j\\ 
1&j^2&j&j^2&j&1
\end{pmatrix}$$

We denote by $\mathbb T$ the unit circle, and we use rectangular matrices over it, with the equivalence relation in Definition 5.3.

\begin{lemma}
Let $h\in M_{3\times 6}(\mathbb T)$ be a matrix having the property that each of the $3$ scalar products between its rows is of the form $x+jx+j^2x+y+jy+j^2y$, for some $x,y\in\mathbb T$. Then modulo equivalence we have either
$$h=\begin{pmatrix}
1&1&1&1&1&1\\
1&j&j^2&r&jr&j^2r\\
1&j^2&j&s&j^2s&js
\end{pmatrix}$$
for some $r,s\in\mathbb T$, or all $18$ entries of $h$ are in $\{1,j,j^2\}$.
\end{lemma}

\begin{proof}
By using the equivalence relation, we may assume that our matrix if of the following form, where the underlined numbers are taken up to permutations:
$$h=\begin{pmatrix}
1&1&1&1&1&1\\
1&j&j^2&r&jr&j^2r\\
1&\underline{j}&\underline{j^2}&\underline{s}&\underline{js}&\underline{j^2s}
\end{pmatrix}$$

We will use several times the procedure consisting in ``using the equivalence relation, plus rescaling the parameters'', to be reffered to as ``arrangement'' of the matrix.

These arrangements will all be done by keeping the first row of $h$ fixed. So, let us denote by $h'$ the matrix formed by the second and third rows of $h$:
$$h'=\begin{pmatrix}
1&j&j^2&r&jr&j^2r\\
1&\underline{j}&\underline{j^2}&\underline{s}&\underline{js}&\underline{j^2s}
\end{pmatrix}$$

We denote by $P$ the scalar product between the two rows of $h'$.

We have 3 cases, depending on how $\underline{j},\underline{j^2}$ are positioned with respect to $j,j^2$.

Case 1: $\underline{j},\underline{j^2}$ are below $j,j^2$. We have two cases here:

Case 1.1: $\underline{j},\underline{j^2}$ are below $j,j^2$, in order. After arrangement, the matrix is:
$$h'=\begin{pmatrix}
1&j&j^2&r&jr&j^2r\\
1&j&j^2&s&\underline{js}&\underline{j^2s}
\end{pmatrix}$$

Since $P=1+1+1+\ldots$, there is no solution here.

Case 1.2: $\underline{j},\underline{j^2}$ are below $j,j^2$, in reverse order. After arrangement, we have:
$$h'=\begin{pmatrix}
1&j&j^2&r&jr&j^2r\\
1&j^2&j&s&\underline{js}&\underline{j^2s}
\end{pmatrix}$$

The solution here is the matrix in the statement.

Case 2: one of $\underline{j},\underline{j^2}$ is below one of $j,j^2$, and the other one isn't. We have two cases:

Case 2.1: $\underline{j}$ is under $j$, or $\underline{j^2}$ is under $j^2$. In the first case, the arranged matrix is:
$$h'=\begin{pmatrix}
1&j&j^2&r&jr&j^2r\\
1&j&s&j^2&\underline{js}&\underline{j^2s}
\end{pmatrix}$$

Thus we must have $r,s\in\{1,j,j^2\}$. The other case, $\underline{j^2}$ under $j^2$, is similar.

Case 2.2: $\underline{j}$ is under $j^2$, or $\underline{j^2}$ is under $j$. By interchanging the second and the third row, we may assume that $\underline{j^2}$ is under $j$. After arrangement, the matrix is:
$$h'=\begin{pmatrix}
1&j&j^2&r&jr&j^2r\\
1&j^2&s&j&\underline{js}&\underline{j^2s}
\end{pmatrix}$$

Once again, we conclude that the 18 entries of $h$ must be in $\{1,j,j^2\}$.

Case 3: $\underline{j},\underline{j^2}$ are not under $j,j^2$. After rescaling $r,s$, we may assume that $\underline{j}$ is under $r$ and that $\underline{s}$ is under $j$, and we have two cases:

Case 3.1: under $j^2$ we have $\underline{js}$. The matrix is:
$$h'=\begin{pmatrix}
1&j&j^2&r&jr&j^2r\\
1&s&js&j&\underline{j^2}&\underline{j^2s}
\end{pmatrix}$$

By examining $P$ we conclude that we have either a particular case of the general solution in the statement, or we are in the situation $r,s\in\{1,j,j^2\}$.

Case 3.2: under $j^2$ we have $\underline{j^2s}$. The matrix is:
$$h'=\begin{pmatrix}
1&j&j^2&r&jr&j^2r\\
1&s&j^2s&j&\underline{j^2}&\underline{js}
\end{pmatrix}$$

Once again, by examining $P$ we conclude that we have either a particular case of the general solution in the statement, or we are in the situation $r,s\in\{1,j,j^2\}$.
\end{proof}

\begin{lemma}
Let $h\in M_{4\times 6}(\mathbb T)$ be a matrix having the property that each of the $6$ scalar products between its rows is of the form $x+jx+j^2x+y+jy+j^2y$, for some $x,y\in\mathbb T$. Then modulo equivalence, all $24$ entries of $h$ are in $\{1,j,j^2\}$.
\end{lemma}

\begin{proof}
We apply Lemma 8.1 to the first three rows, and then we multiply the fourth row by a suitable scalar, as for the matrix to become dephased. We denote by $h'$ the matrix obtained by deleting the first of 1's, which must look as follows:
$$h'=\begin{pmatrix}
1&j&j^2&r&jr&j^2r\\
1&j^2&j&s&j^2s&js\\
1&\underline{j}&\underline{j^2}&\underline{t}&\underline{jt}&\underline{j^2t}
\end{pmatrix}$$

We denote by $P_1,P_2$ the scalar products of the third row with the first two rows, and we use the same conventions as in the proof of the previous lemma.

We have three cases, depending on where $\underline{j},\underline{j^2}$ are positioned:

Case 1: $\underline{j},\underline{j^2}$ are in the second and third column. By symmetry we can assume that $\underline{j},\underline{j^2}$ appear in this order, and we can arrange the matrix as follows:
$$h'=\begin{pmatrix}
1&j&j^2&r&jr&j^2r\\
1&j^2&j&s&j^2s&js\\
1&j&j^2&t&\underline{jt}&\underline{j^2t}
\end{pmatrix}$$

We have $P_1=1+1+1+\ldots$, so there is no solution here.

Case 2: one of $\underline{j},\underline{j^2}$ is in the second or third column, and the other one isn't. By symmetry we can assume that $\underline{j}$ is in the second column, and the arranged matrix is:
$$h'=\begin{pmatrix}
1&j&j^2&r&jr&j^2r\\
1&j^2&j&s&j^2s&js\\
1&j&t&j^2&\underline{jt}&\underline{j^2t}
\end{pmatrix}$$

We have $P_1=1+1+\ldots$, so $P_1$ must be of the form $1+1+j+j+j^2+j^2$, and it follows that we have $r,t\in\{1,j,j^2\}$. In the case $t=j^2$ we get back to Case 1, and we are done. In the case $t\in\{1,j\}$ we have $P_2=1+j+j\bar{t}+\ldots$, with $j\bar{t}\neq j^2$, so the missing $j^2$ term of $P_2$ must come from a scalar product coming from one of the last three columns. But this means that we have $s\in\{1,j,j^2\}$, and we are done again.

Case 3: none of $\underline{j},\underline{j^2}$ is in the second or third column. In this case we can arrange the matrix in the following way:
$$h'=\begin{pmatrix}
1&j&j^2&r&jr&j^2r\\
1&j^2&j&s&j^2s&js\\
1&jt&j^2t&t&\underline{j}&\underline{j^2}
\end{pmatrix}$$

We have $P_1=1+\bar{t}+\bar{t}+\ldots$, so $P_1$ must be of the form $1+1+j+j+j^2+j^2$, and it follows that we have $r,t\in\{1,j,j^2\}$. In the case $t=1$ we get back to Case 1, and we are done. In the case $t\in\{j,j^2\}$ we have $P_2=1+j\bar{t}+j^2\bar{t}\ldots$, with $1\in\{j\bar{t},j^2\bar{t}\}$, so $P_2$ must be of the form $1+1+j+j+j^2+j^2$. Thus $s\in\{1,j,j^2\}$, and we are done.
\end{proof}

\begin{theorem}
The Tao matrix $T\in M_{6\times 6}(\mathbb T)$ is the only complex Hadamard matrix at $n=6$ having the property that all $15$ scalar products between its rows are of the form $x+jx+j^2x+y+jy+j^2y$, for some $x,y\in\mathbb T$.
\end{theorem}

\begin{proof}
We know from Lemma 8.2 that any Hadamard matrix $h$ as in the statement must have all its entries in $\{1,j,j^2\}$. The idea will be to reconstruct this matrix, by starting with the first 2 rows, then by adding 4 more rows, one at a time.

First, by using the equivalence relation, we can assume that the matrix $h_2\in M_{2\times 6}(\mathbb T)$ consisting of the first two rows of $h$ is as follows:
$$h_2=\begin{pmatrix}
1&1&1&1&1&1\\ 
1&1&j&j&j^2&j^2
\end{pmatrix}$$

When trying to add one more row to this matrix, under the assumption in the statement, the solutions modulo equivalence are:
$$h_3=\begin{pmatrix}
1&1&1&1&1&1\\ 
1&1&j&j&j^2&j^2\\ 
1&j&1&j^2&j^2&j
\end{pmatrix}$$
$$h_3=\begin{pmatrix}
1&1&1&1&1&1\\ 
1&1&j&j&j^2&j^2\\ 
1&j^2&j^2&j&1&j
\end{pmatrix}$$

Since the problem is symmetric in $j,j^2$, we may assume that we are in the first case. Now when trying to add a fourth row to this matrix, the solutions are:
$$h_4=\begin{pmatrix}
1&1&1&1&1&1\\ 
1&1&j&j&j^2&j^2\\ 
1&j&1&j^2&j^2&j\\
1&j&j^2&1&j&j^2
\end{pmatrix}$$
$$h_4=\begin{pmatrix}
1&1&1&1&1&1\\ 
1&1&j&j&j^2&j^2\\ 
1&j&1&j^2&j^2&j\\
1&j^2&j^2&j&1&j
\end{pmatrix}$$
$$h_4=\begin{pmatrix}
1&1&1&1&1&1\\ 
1&1&j&j&j^2&j^2\\ 
1&j&1&j^2&j^2&j\\
1&j^2&j&j^2&j&1
\end{pmatrix}$$

Let us try now to construct the full $6\times 6$ matrix. Since the same row cannot be added several times, the above three solutions for the 4-th row are in fact the solutions for the 4-th, 5-th and 6-th row, and we obtain the Tao matrix as claimed.
\end{proof}

\section{The Haagerup matrix}

In this section  we find an abstract characterization of the Haagerup matrix:
$$H^q=\begin{pmatrix}
1&1&1&1&1&1\\
1&-1&i&i&-i&-i\\ 
1&i&-1&-i&q&-q\\ 
1&i&-i&-1&-q&q\\
1&-i&\bar{q}&-\bar{q}&i&-1\\ 
1&-i&-\bar{q}&\bar{q}&-1&i
\end{pmatrix}$$

We denote by $\mathbb T$ the unit circle, and we use rectangular matrices over it, with the equivalence relation in Definition 5.3.

\begin{lemma}
Let $h\in M_{3\times 6}(\mathbb T)$ be a matrix such that each of the $3$ scalar products between its rows is of the form $x-x+y-y+z-z$. Then modulo equivalence we can assume that the first row consists of $1$'s, and the rest of the matrix is of type
$$h_1=\begin{pmatrix}
1&-i&1&i&-1&-1\\
1&-1&i&-i&q&-q
\end{pmatrix}$$
$$h_2=\begin{pmatrix}
1&1&-1&i&-1&-i\\
1&-1&q&-q&iq&-iq
\end{pmatrix}$$
$$h_3=\begin{pmatrix}
1&-1&i&-i&q&-q\\
1&-i&i&-1&-q&q
\end{pmatrix}$$
$$h_4=\begin{pmatrix}
1&-i&-1&i&q&-q\\
1&-1&-q&-iq&iq&q
\end{pmatrix}$$
for some $q\in\mathbb T$.
\end{lemma}

\begin{proof}
We use the various conventions in Lemma 8.1. After assuming that the first row consists of 1's, the rest of the matrix looks as follows:
$$h'=\begin{pmatrix}
1&-1&a&-a&b&-b\\
1&\underline{-1}&\underline{x}&\underline{-x}&\underline{y}&\underline{-y}
\end{pmatrix}$$

We denote by $P$ the scalar product between the rows of $h'$. We have two cases, depending on where the missing $-1$ entry of $P$ comes from. 

Case A: assume first that the missing $-1$ entry of $P$ comes from a product involving the entries $-1$ or $\underline{-1}$. After arrangement, the matrix becomes: 
$$h'=\begin{pmatrix}
1&-1&a&-a&b&-b\\
1&1&-1&x&-1&-x
\end{pmatrix}$$

We have $P=1-1-a-a\bar{x}-b+b\bar{x}$, and the solution is of type $h_2$:
$$h'=\begin{pmatrix}
1&-1&a&-a&ia&-ia\\
1&1&-1&i&-1&-i
\end{pmatrix}$$

Case B: assume now that the missing $-1$ entry of $P$ comes from a product not involving the entries $-1$ or $\underline{-1}$. After arrangement, the matrix becomes:
$$h'=\begin{pmatrix}
1&-1&a&-a&b&-b\\
1&\underline{-1}&\underline{x}&\underline{-x}&\underline{-b}&b
\end{pmatrix}$$

We have 3 cases, depending on where $\underline{-1}$ is located:

Case 1: $\underline{-1}$ is under $-1$. The matrix becomes:
$$h'=\begin{pmatrix}
1&-1&a&-a&b&-b\\
1&-1&\underline{x}&\underline{-x}&\underline{-b}&b
\end{pmatrix}$$

Since $P$ already contains the numbers $1,1,-1$, we have several cases, depending on where the missing number $-1$ comes from, and the solution is of type $h_3$:
$$h'=\begin{pmatrix}
1&-1&a&-a&ia&-ia\\
1&-1&-ia&a&-a&ia
\end{pmatrix}$$

Case 2: $\underline{-1}$ is under $-a$. The matrix becomes: 
$$h'=\begin{pmatrix}
1&-1&a&-a&b&-b\\
1&\underline{-x}&\underline{x}&-1&\underline{-b}&b
\end{pmatrix}$$

Since $P$ already contains the numbers $1,-1,a$, we have several cases, depending on where the missing entry $-a$ comes from. After arrangement, these cases are:

Case 2.1: $-a$ comes from $\underline{-x}$ under $-1$. The solutions are of type $h_3,h_4$:
$$h'=\begin{pmatrix}
1&-1&i&-i&b&-b\\
1&-i&i&-1&-b&b
\end{pmatrix}$$
$$h'=\begin{pmatrix}
1&-1&a&-a&i&-i\\
1&\bar{a}&-i&-1&-\bar{a}&i
\end{pmatrix}$$

Case 2.2: $-a$ comes from $\underline{-b}$ under $-1$. The solution is of type $h_1$:
$$h'=\begin{pmatrix}
1&-1&i&-i&i&-i\\
1&-i&x&-1&-x&i
\end{pmatrix}$$

Case 2.3: $-a$ comes from $\underline{x}$ under $a$. The solution is of type $h_1$:
$$h'=\begin{pmatrix}
1&-1&a&-a&i&-i\\
1&-i&-1&-1&1&i
\end{pmatrix}$$

Case 2.4: $-a$ comes from $\underline{-x}$ under $-b$. The solution is of type $h_4$:
$$h'=\begin{pmatrix}
1&-1&i&-i&b&-b\\
1&-b&-ib&-1&ib&b
\end{pmatrix}$$

Case 3: $\underline{-1}$ is under $b$. The matrix becomes: 
$$h'=\begin{pmatrix}
1&-1&a&-a&b&-b\\
1&\underline{x}&\underline{-x}&\underline{-b}&-1&b
\end{pmatrix}$$

Since $P$ already contains the numbers $1,-1,-b$, we have several cases, depending on where the missing entry $b$ comes from. After arrangement, these cases are:

Case 3.1: $b$ comes from $\underline{-x}$ under $-1$. The solution is of type $h_1$:
$$h'=\begin{pmatrix}
1&-1&a&-a&i&-i\\
1&i&-i&-i&-1&i
\end{pmatrix}$$

Case 3.2: $b$ comes from $\underline{x}$ under $a$. The solutions are of type $h_4,h_3$:
$$h'=\begin{pmatrix}
1&-1&a&-a&ia&-ia\\
1&i&-i&-ia&-1&ia
\end{pmatrix}$$
$$h'=\begin{pmatrix}
1&-1&a&-a&i&-i\\
1&-i&-ia&ia&-1&i
\end{pmatrix}$$

Case 3.3: $b$ comes from $\underline{-b}$ under $a$. The solution is of type $h_1$:
$$h'=\begin{pmatrix}
1&-1&1&-1&i&-i\\
1&x&-i&-x&-1&i
\end{pmatrix}$$

Case 3.4: $b$ comes from $\underline{-b}$ under $-a$. The solution is of type $h_1$:
$$h'=\begin{pmatrix}
1&-1&-1&1&i&-i\\
1&x&-x&-i&-1&i
\end{pmatrix}$$

This finishes the proof.
\end{proof}

\begin{theorem}
The Haagerup matrix $H^q\in M_{6\times 6}(\mathbb T)$ with $q\in\mathbb T$ is the only complex Hadamard matrix at $n=6$ having the property that all $15$ scalar products between its rows are of the form $x-x+y-y+z-z$, for some $x,y,z\in\mathbb T$.
\end{theorem}

\begin{proof}
Let $h$ be a matrix as in the statement, assumed to be dephased.

By applying Lemma 9.1 to all the $3\times 6$ submatrices of $h$, we deduce that all the entries of $h$ are in $\{\pm 1,\pm i,\pm q,\pm iq\}$, for some $q\in\mathbb T$. 

Moreover, from the structure of the explicit solutions in Lemma 9.1, we deduce that the rows can fall into 3 classes, depending on number of $q$'s, which can be $0,2,4$.

We also know from Lemma 9.1 that the $0,2,4$ possible $q$ parameters on different rows can overlap vertically on $0$ or $2$ positions. This leads to the conclusion that our matrix has a $3\times 3$ block decomposition, of the following form:
$$h=\begin{pmatrix}
A&B&C\\
D&xE&yF\\
G&zH&tI
\end{pmatrix}$$

Here $A,\ldots,I$ are $2\times 2$ matrices over $\{\pm 1,\pm i\}$, and $x,y,z,t$ are in $\{1,q\}$. A more careful examination shows that the solution must be of the following form:
$$h=\begin{pmatrix}
A&B&C\\
D&E&qF\\
G&qH&qI
\end{pmatrix}$$

More precisely, the matrix must be as follows:
$$h=\begin{pmatrix}
1&1&1&1&1&1\\
1&1&-i&i&-1&-1\\ 
1&i&-1&-i&-q&q\\ 
1&-i&i&-1&-iq&iq\\
1&-1&q&-iq&iq&-q\\ 
1&-1&-q&iq&q&-iq
\end{pmatrix}$$

By multiplying the rows by suitable scalars, we have:
$$h=\begin{pmatrix}
1&1&1&1&1&1\\
i&i&1&-1&-i&-i\\ 
-1&-i&1&i&q&-q\\ 
-i&-1&1&i&-q&q\\
1&-1&q&-iq&iq&-q\\ 
-1&1&q&-iq&-q&iq
\end{pmatrix}$$

By permuting the first two columns with the middle two columns, we get:
$$h=\begin{pmatrix}
1&1&1&1&1&1\\
1&-1&i&i&-i&-i\\ 
1&i&-1&-i&q&-q\\ 
1&i&-i&-1&-q&q\\
q&-iq&1&-1&iq&-q\\ 
q&-iq&-1&1&-q&iq
\end{pmatrix}$$

But this is precisely the Haagerup matrix with the last two rows multiplied by $q$, and we are done.
\end{proof}

\section{Di\c t\u a deformations}

We know from the previous sections that the Tao and Haagerup matrices $T$ and $H^q$ are uniquely determined among the $6\times 6$ complex Hadamard matrices by the nature of the $15$ scalars products between the rows. For the Tao matrix all these scalar products are of the form $x+jx+jx^2+y+jy+jy^2$, with $j=e^{2\pi i/3}$, and for the Haagerup matrix these scalar products are of the form $x-x+y-y+z-z$.

In this section we investigate the ``mixed'' case, where both types of scalar products appear. We will show that the only solutions are the Di\c t\u a deformations of $F_6$.

We have two proofs for this result, none of which is really satisfactory. The first proof is based on a number of ``reductions'' of arithmetic nature, basically asserting that: (1) in order to classify the regular matrices we can restrict attention to the regular matrices in the Butson class, and (2) in order to classify the regular Butson matrices at $n=6$ we can restrict attention to the matrices in $H_6(30)$. This latter problem can be solved by a computer, and the solutions that we found are indeed the two Di\c t\u a deformations of $F_6$. However, the arithmetic reduction part is quite delicate to justify, and the use of the program at the end is not very satisfactory. We intend to explain, refine and generalize this approach in some future systematic work on the regular matrices.

The second proof that we have is in the spirit of those given in the previous two sections, with the important difference, however, that it is much more complex. The point is that the ``mixed'' case requires a whole sequence of lemmas in the spirit of Lemma 8.1, Lemma 8.2 and Lemma 9.1, basically one for each possible configuration, from the point of view of the scalar products, of matrices having 3 or 4 rows. 

In what follows we will present the main ideas of this second proof, by skipping a number of technical details. We begin with some definitions.

\begin{definition}
Let $P=<u,v>$ be a scalar product, with $u,v\in\mathbb T^6$.
\begin{enumerate}
\item We say that $P$ is binary if it is of the form $x-x+y-y+z-z$.

\item  We say that $P$ is ternary if it is of the form $x+jx+jx^2+y+jy+jy^2$.
\end{enumerate}
\end{definition}

Assume now that we have a ``mixed'' matrix $h\in M_6(\mathbb T)$, in the sense that all $15$ scalars products between rows are binary or ternary, and that both the binary and ternary cases appear. We associate to $h$ a colored graph $X$, in the following way: $X$ is the complete 6-graph having as vertices the rows of $h$, and each edge is colored 2 or 3, depending on whether the corresponding scalar product is binary or ternary.

\begin{lemma}
Let $h\in M_6(\mathbb T)$ be a mixed matrix, having row graph $X$.
\begin{enumerate}
\item $X$ has no binary triangle.

\item $X$ has no ternary square.

\item $X$ has at least one ternary triangle.
\end{enumerate}
\end{lemma}

\begin{proof}
This result follows from the lemmas in the previous sections:

(1) Assume that $X$ has a binary triangle. By arranging the matrix, we may assume that the 3 scalar products between the first 3 rows of $h$ are binary, and that the 4-th row has at least one ternary scalar product with the first 3 rows, say with the first one. We can apply Lemma 9.1 to the matrix formed by the first 3 rows, and a case-by-case analysis shows that we cannot complete this matrix with a 4-th row as above.

(2) Assume that $X$ has a ternary square. By arranging the matrix, we may assume that the 6 scalar products between the first 4 rows of $h$ are ternary, and that the 5-th row has at least one binary scalar product with the first 4 rows, say with the first one.

We can apply Lemma 8.2 to the matrix formed by the 4 rows, and a case-by-case analysis shows that we cannot complete this matrix with a 5-th row as above.

(3) Assume that $X$ has no ternary triangle. By using (1) we conclude that all the triangles are ``mixed'', and together with (2) this shows that we have only 2 possibilities for the squares. By looking now at pentagons, we see that only one case is possible, namely the usual pentagon with edges colored 2, with the stellar pentagon formed by the diagonals with edges colored 3. Since it is impossible to complete this pentagon to a hexagon, as for all triangles to be ``mixed'', we are done.
\end{proof}

In order to start the classification, the idea would be to assume that the first three rows form a ternary triangle, to apply Lemma 8.1, that to try to complete the matrix with a 4-th row. In order to do so, we will need one more technical lemma.

\begin{lemma}
There is no mixed matrix $h\in M_{4\times 6}(\mathbb T)$ having the following properties:
\begin{enumerate}
\item The first $3$ rows have ternary scalar products between them.

\item The $4$-th row has exactly $2$ binary products with the first $3$ rows.
\end{enumerate}
\end{lemma}

\begin{proof}
We know from Lemma 8.1 that the matrix must look as follows:
$$h=\begin{pmatrix}
1&1&1&1&1&1\\
1&j&j^2&r&jr&j^2r\\
1&j^2&j&s&j^2s&js\\
1&\underline{j}&\underline{j^2}&\underline{t}&\underline{jt}&\underline{j^2t}
\end{pmatrix}$$

The scalar products of the fourth row with the second and third row are both binary, and an examination of all the possible cases shows that this is not possible.
\end{proof}

We are now in position of stating a key result.

\begin{proposition}
The row graph of a mixed matrix $h\in M_6(\mathbb C)$ can be:
\begin{enumerate}
\item Either the bipartite graph having $3$ binary edges.

\item Or the bipartite graph having $2$ ternary triangles.
\end{enumerate}
\end{proposition}

\begin{proof}
Let $X$ be the row graph in the statement.

By using Lemma 10.2 and Lemma 10.3, we see that there are only two types of squares: (1) those having 1 binary edge and 5 ternary edges, and (2) those consisting of a ternary triangle, connected to the 4-th point with 3 binary edges.

By looking at pentagons, then hexagons that can be built with these squares, we see that the above two types of squares cannot appear at the same time, at that at the level of hexagons, we have the two solutions in the statement. 
\end{proof}

We will show now that the dichotomy produced by Proposition 10.4 corresponds in fact to the two possible Di\c t\u a deformations of $F_6$, coming from $6=2\times 3=3\times 2$.

As explained in section 6, when constructing a Di\c t\u a deformation we can always assume that the matrix of parameters has 1 on the first row and column. Thus the Di\c t\u a deformations of $F_2\otimes F_3$ are the following matrices:
$$F_{23}^{rs}
=\begin{pmatrix}1&1\\ 1&-1\end{pmatrix}
\otimes_{\begin{pmatrix}1&1\\ 1&r\\ 1&s\end{pmatrix}}
\begin{pmatrix}1&1&1\\ 1&j&j^2\\ 1&j^2&j\end{pmatrix}$$

In  Di\c t\u a product notation, this matrix is:
$$F_{23}^{rs}
=\begin{pmatrix}1&1\\ 1&-1\end{pmatrix}
\otimes\left(
\begin{pmatrix}1&1&1\\ 1&j&j^2\\ 1&j^2&j\end{pmatrix},
\begin{pmatrix}1&1&1\\ r&jr&j^2r\\ s&j^2s&js\end{pmatrix}
\right)$$

Thus we have the following formula:
$$F_{23}^{rs}
=\begin{pmatrix}
1&1&1&1&1&1\\
1&j&j^2&r&jr&j^2r\\
1&j^2&j&s&j^2s&js\\ 
1&1&1&-1&-1&-1\\
1&j&j^2&-r&-jr&-j^2r\\
1&j^2&j&-s&-j^2s&-js
\end{pmatrix}$$

As for the Di\c t\u a deformations of $F_3\otimes F_2$, these are the following matrices:
$$F_{32}^{rs}
=\begin{pmatrix}1&1&1\\ 1&j&j^2\\ 1&j^2&j\end{pmatrix}
\otimes_{\begin{pmatrix}1&1&1\\1&r&s\end{pmatrix}}
\begin{pmatrix}1&1\\ 1&-1\end{pmatrix}$$

In Di\c t\u a product notation, we have:
$$F_{32}^{rs}
=\begin{pmatrix}1&1&1\\ 1&j&j^2\\ 1&j^2&j\end{pmatrix}
\otimes\left(
\begin{pmatrix}1&1\\ 1&-1\end{pmatrix},
\begin{pmatrix}1&1\\ r&-r\end{pmatrix},
\begin{pmatrix}1&1\\ s&-s\end{pmatrix}
\right)$$

Thus we have the following formula:
$$F_{32}^{rs}
=\begin{pmatrix}
1&1&1&1&1&1\\
1&-1&r&-r&s&-s\\
1&1&j&j&j^2&j^2\\ 
1&-1&jr&-jr&j^2s&-j^2s\\
1&1&j^2&j^2&j&j\\
1&-1&j^2r&-j^2r&js&-js
\end{pmatrix}$$

Observe that, modulo equivalence, $F_{32}^{rs}$ is nothing but the transpose of $F_{23}^{rs}$. This comes in fact from a general property of Di\c t\u a deformations, not to be detailed here.

\begin{theorem}
The two Di\c t\u a deformations of $F_6$ are the unique Hadamard matrices having the property that all $15$ scalar products between rows are of the form $x-x+y-y+z-z$ or the form $r+jr+j^2r+s+js+j^2s$, with both cases appearing.
\end{theorem}

\begin{proof}
We apply Proposition 10.4, and we have two cases:

(1) Assume first that the row graph is the bipartite one with 3 binary edges. By permuting the rows, we can assume that the binary scalars products are those between rows $i$ and $i+3$. By applying Lemma 8.1 to the first three rows, and also to the second, third and fourth rows, we get that the matrix formed by the 4 first rows is of the form:
$$h_4=\begin{pmatrix}
1&1&1&1&1&1\\
1&j&j^2&r&jr&j^2r\\
1&j^2&j&s&j^2s&js\\ 
1&1&1&t&t&t
\end{pmatrix}$$

Now since the scalar product between the first and the fourth row is binary, we must have $t=-1$, so the solution is:
$$h_4=\begin{pmatrix}
1&1&1&1&1&1\\
1&j&j^2&r&jr&j^2r\\
1&j^2&j&s&j^2s&js\\ 
1&1&1&-1&-1&-1
\end{pmatrix}$$

We can use the same argument for finding the fifth and sixth row, by arranging the matrix formed by the first three rows such as the second, respectively third row consist only of 1's. This arrangement will make appear some parameters of the form $j,j^2,r,s$ in the extra row, and we obtain as unique solution the Di\c t\u a deformation $F_{23}^{rs}$. 

(2) Assume now that the row graph is the bipartite one with 2 ternary triangles. By permuting the rows, we can assume that the ternary triangles are those formed by the first three rows, and by the last three rows. Let us look now at the matrix formed by the first four rows. By using Lemma 8.1, this matrix must be of the following form:
$$h_4=\begin{pmatrix}
1&1&1&1&1&1\\
1&j&j^2&a&ja&j^2a\\
1&j^2&j&b&j^2b&jb\\ 
1&-1&\underline{r}&\underline{-r}&\underline{s}&\underline{-s}
\end{pmatrix}$$

Our assumption is that the scalar products of the fourth row with the second and third rows are binary, and a case-by-case analysis shows that we must have $a,b\in\{1,j,j^2\}$, and that the solution is of the following type: 
$$h_4=\begin{pmatrix}
1&1&1&1&1&1\\
1&1&j&j&j^2&j^2\\ 
1&1&j^2&j^2&j&j\\
1&-1&r&-r&s&-s
\end{pmatrix}$$

We can use the same argument for finding the fifth and sixth row, and we conclude that the matrix is of the following type:
$$h=\begin{pmatrix}
1&1&1&1&1&1\\
1&1&j&j&j^2&j^2\\ 
1&1&j^2&j^2&j&j\\
1&-1&r&-r&s&-s\\
1&-1&a&-a&b&-b\\
1&-1&c&-c&d&-d
\end{pmatrix}$$

Now since the last three rows must form a ternary triangle, we conclude that the matrix must be of the following form:
$$h=\begin{pmatrix}
1&1&1&1&1&1\\
1&1&j&j&j^2&j^2\\ 
1&1&j^2&j^2&j&j\\
1&-1&r&-r&s&-s\\
1&-1&jr&-jr&j^2s&-j^2s\\
1&-1&j^2r&-j^2r&js&-js
\end{pmatrix}$$

By permuting the rows we get the Di\c t\u a deformation $F_{32}^{rs}$, and we are done.
\end{proof}

\section{Classification results}

We are now in position of stating the main results in this paper. We will combine the abstract Hopf algebra results in section 6 with the Butson matrix philosophy from section 7, and with the various classification results in sections 8-10. 

We have first the following key definition.

\begin{definition}
A complex Hadamard matrix is called regular if all the scalar products between distinct rows decompose as sums of cycles.
\end{definition}

Here by ``cycle'' we mean of course cycle in a generalized sense, i.e. the sum of the $p$-roots of unity, with $p\in\mathbb N$ prime, rotated by an arbitrary scalar $a\in\mathbb T$:
$$C=ae^{2\pi i/p}+ae^{4\pi i/p}+\ldots+ae^{2(p-1)\pi i/p}$$

As mentioned in section 7, all the known examples of Butson matrices are regular, and we conjecture that the regularity condition is automatic in the Butson case.

Observe also that all the explicit matrices given in this paper are regular, except for the Bj\"orck-Fr\"oberg matrix. In fact, at $n=6$, there are several quite mysterious classes of complex Hadamard matrices, all non-regular. See \cite{ben}, \cite{szo}, \cite{tz1}.

We have the following result.

\begin{theorem}
The regular complex Hadamard matrices at $n=6$ are as follws:
\begin{enumerate}
\item Tao matrix $T$.

\item Haagerup matrix $H^q$.

\item Di\c t\u a deformations $F_{23}^{rs}$.

\item Di\c t\u a deformations $F_{32}^{rs}$.
\end{enumerate}
\end{theorem}

\begin{proof}
The equation $x_1+\ldots+x_6=0$ with $x_i\in\mathbb T$ has two types of regular solutions: those consisting of three 2-cycles, and those consisting of two 3-cycles.

(1) In case all the 15 scalar products consist of two 3-cycles, we know from Theorem 8.3 that the only solution is the Tao matrix $T$.

(2) In case all the 15 scalar products consist of three 2-cycles, we know from Theorem 9.2 that the only solution is the Haagerup matrix $H^q$.

(3) In case some of the 15 scalar products consist of two 3-cycles, and some other consist of three 2-cycles, we know from Theorem 10.5 that the only solutions are the Di\c t\u a  deformations of $F_2\otimes F_3$ and of $F_3\otimes F_2$.
\end{proof}

As a first consequence, we obtain another general result at $n=6$.

\begin{theorem}
The regular Butson matrices at $n=6$ are as follows:
\begin{enumerate}
\item Tao matrix $T$.

\item Haagerup matrix $H^q$, with $q$ root of unity.

\item Di\c t\u a deformations $F_{23}^{rs}$, with $r,s$ roots of unity.

\item Di\c t\u a deformations $F_{32}^{rs}$, with $r,s$ roots of unity.
\end{enumerate}
\end{theorem}

\begin{proof}
This follows from Theorem 11.2.
\end{proof}

We should mention that the regularity condition being conjecturally automatic for the Butson matrices, this type of result covers in principle all the Butson matrices. In the particular case of the above result, we can actually prove that the regularity condition is automatic at $n=6$, but the details won't be given here. The idea is that the ``tricky sum'' described in section 7 can be excluded by a computer program.

We can state now the main result in this paper.

\begin{theorem}
The quantum permutation algebras associated to the regular Hadamard matrices at $n\leq 6$ are as follows:
\begin{enumerate}
\item The algebras $C(\mathbb Z_2)$, $C(\mathbb Z_3)$, $C(\mathbb Z_5)$.

\item Quotients of $C(\mathbb Z_2)*_wC(\mathbb Z_2)$.

\item Quotients of $C(S_3)*_wC(\mathbb Z_2)$.

\item Quotients of $C(\mathbb Z_2)*_wC(S_3)$.

\item The algebras associated to $T$, $H^q$.
\end{enumerate}
\end{theorem}

\begin{proof}
This follows indeed by combining the various results in Theorem 1.6, Theorem 5.4, Theorem 6.8 and Theorem 11.2.
\end{proof}

As a first comment, the algebras in (2) are explicitely computed in \cite{bni}. They all appear as twists of group algebras of type $C^*(\Gamma)$, with $\Gamma$ quotient of $D_\infty$.

In principle the algebras in (3,4) can be investigated by using similar methods. The main problem here is the computation of the generic algebra, and this is in relation with the general question formulated at the end of section 6.

Regarding now the algebras in (5), these rather seem to be of ``exceptional'' nature. This is particularly true for the algebra associated to the Tao matrix $T$, which is known to be isolated \cite{tz2}. The algebra associated to $H^q$, however, has a different status, because the matrices $H^q$ form an affine family in the sense of \cite{tz2}.

\begin{problem}
What is the Hopf algebra associated to the Haagerup matrix $H^q$, for generic values of the parameter?
\end{problem}

The point here is that a systematic investigation of the affine regular case seems to be a key problem. At $n=7$ indeed we have the Petrescu matrix $P^q$, where the computation of the generic algebra corresponds to a well-known problem in subfactor theory, of potential interest in connection with several questions raised by \cite{bnp}, \cite{jo2}.

\section{Concluding remarks}

We have seen in this paper that the Hopf image approach to the quantum permutation algebras leads to a natural hierarchy of the various ``magic-type'' objects associated to the Hilbert spaces. This hierarchy, while constructed quite abstracly, turns to have the Hadamard matrices at its core, and is therefore in tune with some key problems in combinatorics and quantum physics. Moreover, the representation theory invariants of the Hopf algebra themselves correspond to some subtle subfactor invariants, coming from the work of Jones \cite{jo2} and Popa \cite{po2}, and from this point of view, our hierarchy is once again compatible with some key problems in subfactor theory, notably with the computation of quantum invariants of the Petrescu matrix \cite{pet}.

In view of a further development of this approach, a number of explicit questions were raised in the previous sections. Probably the most important one is the question about the generic algebra for the Di\c t\u a deformations. This question belongs to the general representation theory problematics for the free wreath products, and the conclusion here is that the conjectural statements in \cite{bb1} would have not only to be proved, but also to be substantially refined. There seems to be a lot of work to be done here, and we intend to come back to these questions in some future work.

Finally, let us mention that what is also missing to our quantum permutation group approach to the complex Hadamard matrices are some tools coming from classical analysis. As explained in \cite{bc1}, \cite{bc2}, some fruitful connections with Voiculescu's free probability \cite{vdn}, and with analysis in general, can be found via Weingarten functions, so the main problem is to understand these functions in the general context of Hopf images. Once again, we intend to come back to these questions in some future work.

\end{document}